\documentclass[11pt]{amsart}
\usepackage{tikz-cd}
\usepackage{amsmath}
\usepackage{amssymb}
\usepackage{mathrsfs}
\usepackage{mathtools}
\usepackage[all]{xy}
\usepackage{quiver}
\usepackage{biblatex}
\usepackage{stix}

\newtheorem*{rep@theorem}{\rep@title}
\newcommand{\newreptheorem}[2]{%
\newenvironment{rep#1}[1]{%
 \def\rep@title{#2 \ref{##1}}%
 \begin{rep@theorem}}%
 {\end{rep@theorem}}}

\newtheorem{theorem}{Theorem}[section]
\newtheorem{claim}[theorem]{Claim}

\newtheorem{lemma}[theorem]{Lemma}

\newtheorem{proposition}[theorem]{Proposition}
\newtheorem{corollary}[theorem]{Corollary}
\newreptheorem{corollary}{Corollary}

\newtheorem*{theorem*}{Theorem}
\newreptheorem{theorem}{Theorem}

\theoremstyle{definition}
\newtheorem{definition}[theorem]{Definition}
\newtheorem*{definition*}{Definition}
\newreptheorem{definition}{Definition}
\newtheorem{example}[theorem]{Example}

\newtheorem{question}[theorem]{Question}
\newtheorem*{question*}{Question}
\newreptheorem{question}{Question}

\theoremstyle{remark}
\newtheorem{remark}[theorem]{Remark}

\newenvironment{claimproof}[1]{\par\noindent\emph{Proof of claim.}\space#1}{\hfill $\dashv$\\}

\def\l{{\langle}}
\def\r{{\rangle}}

\newcount\skewfactor
\def\mathunderaccent#1#2 {\let\theaccent#1\skewfactor#2
\mathpalette\putaccentunder}
\def\putaccentunder#1#2{\oalign{$#1#2$\crcr\hidewidth
\vbox to.2ex{\hbox{$#1\skew\skewfactor\theaccent{}$}\vss}\hidewidth}}

\def\smallbox#1{\leavevmode\thinspace\hbox{\vrule\vtop{\vbox
   {\hrule\kern1pt\hbox{\vphantom{\tt/}\thinspace{\tt#1}\thinspace}}
   \kern1pt\hrule}\vrule}\thinspace}

\DeclareMathOperator{\id}{id}

\DeclareMathOperator{\GCH}{GCH}

\DeclareMathOperator{\ZFC}{ZFC}
\DeclareMathOperator{\add}{Add}
\DeclareMathOperator{\Gal}{Gal}
\DeclareMathOperator{\rGal}{Gal^{*}}
\DeclareMathOperator{\Fine}{Fine}
\newcommand{\cf}{{\rm cf}}
\def\Ult{{\rm Ult}}

\newcommand{\res}{\upharpoonright}

\title{Supercompact Measures and the Galvin Property}
\author{Tom Benhamou}
\thanks{The research of the first author was supported by the National Science Foundation under Grant
No. DMS-2246703}
\address[Benhamou]{Department of Mathematics, Rutgers University, New Brunswick, NJ, USA}
\email{tom.benhamou@rutgers.edu}
\author{Ben-Zion Weltsch}
\address[Weltsch]{Department of Mathematics, Rutgers University, New Brunswick, NJ, USA}
\email{ben.w@rutgers.edu}

\begin{document}

\begin{abstract}
    We study saturation properties of $\sigma$-complete measures on $P_\kappa(\lambda)$, where $\lambda$ can be either regular or singular. 
    In particular, we prove that in contrast to Galvin's theorem, the Galvin property from \cite{tomshimonalejandro} fails for normal fine ultrafilters on $P_\kappa(\lambda)$,  answering a question of the first author and Goldberg from \cite{Benhamou_Goldberg_2024}.
    We then provide several applications of our results: to ultrafilters on successors under $UA$, we generalize a result of Gitik regarding density of ground model sets in supercompact Prikry extensions, and to generating sets of $P_\kappa(\lambda)$ measures. In the second part of the paper, we study variations of the Galvin property suitable for ultrafilters over $P_\kappa(\lambda)$, and generalize a result of Foreman-Magidor-Zeman \cite[Thm. 1.2]{foremanmagidorzeman} on determinacy of filter games to the two-cardinal setting, answering a question of the first author and Gitman from \cite{tomvika}.
\end{abstract}

\maketitle

\section*{Introduction}
The motivation for this paper arises from a theorem of F. Galvin (see Theorem~\ref{thm: galvin theorem}), first published in \cite{baumgartner-hajnal-mate}, which states that whenever $\kappa^{<\kappa} = \kappa$, $F$ is a normal filter over $\kappa$, and $\l A_i : i < \kappa^+ \r $ is a sequence of sets in $F$, there is some $I$ of size $\kappa$ such that $\bigcap_{i \in I} A_i \in F$.
For example, Galvin's theorem applies if $F$ is the club filter on $\kappa$ and $\GCH$ holds, or $F$ is a normal measure over a measurable cardinal $\kappa$.
Parameterizing the combinatorial property from Galvin's theorem yields the following definition:
\begin{definition*}\label{def: galvin property}[The Galvin Property]
    Let $X$ be a set and $F$ a filter on $X$ and $\kappa\leq \lambda$ cardinals.
    We say $\Gal(F,\kappa,\lambda)$ holds iff whenever $\l A_i : i < \lambda \r$ is a sequence of sets in $F$ of length $\lambda$, there is some $I \in [\lambda]^\kappa$ such that $\bigcap_{i \in I} A_i \in F$.
\end{definition*}
This is a regularity-like property in the sense of Keisler \cite{ChangKeisler} that has been considered in other contexts as well. Kanamori \cite{cohesiveKanamori} called it \textit{cohesiveness}, while Tukey and Isbell \cite{Tukey,Isbell} and more recently in work of Todorcevic, Dobrinen, Milovich, and the first author \cite{TodorcevicDirSets85,Dobrinen/Todorcevic11,Milovich08,BenhamouDobrinen} used it in connection to the Tukey order.  

Interest in the Galvin property  has had a recent resurgence due to its relevance to Prikry-type forcing theory \cite{gitikdensity,Parttwo,BENHAMOU_GITIK_2024}.
Gitik \cite{gitikdensity} used the Galvin property to prove the following:
\begin{theorem*}[Gitik, \cite{gitikdensity}]
    Suppose $2^\kappa = \kappa^+$ and $U$ is a normal ultrafilter on $\kappa$.
    Then in a generic extension by Prikry forcing with $U$, every  subset of $\kappa^+$ such that contains a ground model subset of the size $\kappa$.
\end{theorem*}
In \cite{BENHAMOU_GITIK_2024} it was then proven that the Galvin property of a general $\kappa$-complete ultrafilter over $\kappa$ is equivalent to the density of ground model sets of the tree-Prikry forcing analogous to the one described in Gitik's theorem.
The first author and Gitik further extend this analysis in \cite{BENHAMOU_GITIK_2024} to study ultrafilters $U$ such that Prikry forcing with $U$ adds a generic for the Cohen forcing at $\kappa$.

In \cite{Benhamou_Goldberg_2024} the first author and Goldberg use a $\diamondsuit$-like principle, which also inspired the one introduced in \cite{tomfanxin}, to show that under the Ultrapower Axiom (UA) if $U$ is a $\sigma$-complete ultrafilter on $\kappa^+$ then $\Gal(U,\kappa^+,\kappa^{++})$ must fail.
In the same paper, they ask the following question:
\begin{question*}(\cite[Question 7.4]{Benhamou_Goldberg_2024})
    Suppose $U$ is a normal, fine ultrafilter on $P_\kappa(\kappa^+)$.
    Must $\Gal(U,\kappa,(2^\kappa)^+)$ hold?
\end{question*}
Here we answer this question negatively.
In fact, not only does the Galvin property need not hold, it never holds in this case:
\begin{reptheorem}{theorem: normal nongalvin}
        Let $U$ be a normal, fine ultrafilter on $P_\kappa(\lambda)$.
    Suppose that $2^{<\lambda}=\lambda$ and $\cf(\lambda) \neq \kappa$.
    Then $\lnot \Gal(U,\kappa,2^\lambda)$.
\end{reptheorem}

The situation when $\cf(\lambda) = \kappa$ is more subtle.
In this case, the proof of Galvin's original theorem generalizes using tools of Goldberg's book \cite{Goldberg+2022}.
\begin{reptheorem}{theorem: cf(kappa) galvin}
    Let $U$ be a normal, fine ultrafilter on $P_\kappa(\lambda)$.
    Suppose that $2^{<\lambda}=\lambda$ and $\cf(\lambda) = \kappa$.
    Then $\Gal(U,\kappa,2^\lambda)$. 
\end{reptheorem}
We then investigate the extent to which this theorem can be improved, and show the following:
\begin{reptheorem}{theorem: cf kappa nongalvin}
    Let $U$ be a normal, fine ultrafilter on $P_\kappa(\lambda)$.
    Suppose that $2^{<\lambda}=\lambda$ and $\cf(\lambda) = \kappa$.
    $\lnot \Gal(U,(2^\kappa)^+,2^\lambda)$.
\end{reptheorem}
This leaves (at least) one case open, namely $\Gal(U,2^\kappa,2^\lambda)$ (see Question \ref{question: cf kappa galvin}).

Our results have several applications to the theory of ultrafilters.
First, the failure of Galvin's property for $\cf(\lambda) \neq \kappa$ shows that normal, fine ultrafilters on $P_\kappa(\lambda)$ are $\kappa$-Tukey-top, namely Tukey-maximal among $\kappa$-complete ultrafilters $U$ over $P_\kappa(\lambda)$, whereas Galvin's theorem says that normal ultrafilters on cardinals are necessarily non-$\kappa$-Tukey-top.

Secondly, we can improve the Theorem from \cite{Benhamou_Goldberg_2024} regarding ultrafilter on successor cardinals under UA:
\begin{repcorollary}{Cor: UA}
    (UA) Let $U$ be a $\sigma$-complete ultrafilter on $\kappa^+$.
    If $2^\kappa=\kappa^+$ then $\lnot \Gal(U,\kappa,2^{\kappa^+})$.
\end{repcorollary}

Then, using our theorem, we generalize Gitik's result on the density of old sets in generic extensions by Prikry forcing to the case of the supercompact Prikry forcing:
\begin{reptheorem}{theorem: density old sets}
    Suppose $U$ is a normal $P_\kappa(\lambda)$ ultrafilter where $\cf(\lambda) \neq \kappa$ and $2^{<\lambda} = \lambda$.
    Let $G$ be generic for the supercompact Prikry forcing with respect to $U$.
    In $V[G]$, there is a set $S \subseteq 2^\lambda$ that contains no ground model set of size $\kappa$.
\end{reptheorem}
Finally, we use these results to study generating sets of $P_\kappa(\lambda)$-measure and conclude the following:
\begin{repcorollary}{Cor: large ultrafilter number}
    If $U$ is a $\sigma$-complete $(\kappa,\lambda)$-regular ultrafilter over $\lambda$, $\kappa$ being a strong limit and $2^{<\lambda}=\lambda$, then $\chi(U)=2^\lambda$.
\end{repcorollary}

In the last section of this paper, we consider the revised Galvin property defined using inclusion modulo the \emph{fine filter} which is the filter  generated by cones, i.e. sets of the form $\check{\alpha} = \{x \in P_\kappa(\lambda) : \alpha \in x\}$. 
In Theorem \ref{thm: wierd Galvin} prove a variation of Galvin's theorem using inclusion modulo this filter.
We then use the inclusion modulo the fine filter to study variation of the filter games introduced in \cite{welchschlict} by Nielsen and Welch, which were used to characterize small large cardinals.
Foreman, Magidor, and Zeman in \cite{foremanmagidorzeman} call these games \emph{Welch games} and use their determinacy to construct interesting ideals.
In particular, they prove:
\begin{theorem*}[\cite{foremanmagidorzeman}, Theorem 1.4]\label{theorem: fmz}
    Assume $2^\kappa = \kappa^+$ and $\kappa$ does not carry a saturated ideal. Let $\gamma$ be an infinite regular cardinal below $\kappa^+$.
    If player II has a winning strategy in the Welch game of length $\gamma$, then there is a uniform normal ideal $I$ on $\kappa$ with a dense set $D \subseteq I^+$ such that:
    \begin{enumerate}
        \item $(D, \subseteq^*)$ is a downwards-growing tree of height $\gamma$
        \item $D$ is closed under $\subseteq_I$-decreasing sequences of length $\gamma$
        \item $D$ is dense in $P(\kappa)/I$.
    \end{enumerate}
\end{theorem*}
From the ideal $I$, a precipitous ideal can be constructed, given that $\gamma>\omega$, showing the equiconsistency of a winning strategy for Player II in the game of length $\omega+1$ and the existence of a measurable cardinal.

 This result was generalized to some extent by the second author and Gitman in \cite{tomvika} to filter games in the two-cardinal settings. Also, it was noticed in \cite{tomvika} that an additional property of the constructed ideal, named \textit{$\mu$-measuring}\footnote{We say that an ideal $I$ is $\mu$-measuring if given any collection $\mathcal{A}=\{A_i\mid i<\mu\}\subseteq P(\kappa)$, and a positive set $B\in I^+$, there is $B'\subseteq B$, $B'\in I^+$ such that for every $i<\mu$, $B'\subseteq_I A_i$ or $B'\subseteq_I \kappa\setminus A_i$.} could be extracted, with which Theorem~\ref{theorem: fmz} could be reversed. However, the generalization was not fully satisfactory as the dense tree did not consist of positive sets from the constructed ideal. In an attempt to find a direct generalization to Theorem~\ref{theorem: fmz}, it was then asked:
\begin{question*}[{\cite[Question 9.8]{tomvika}}] Is there an equivalent version of the two-cardinal games where
sets are played instead of ultrafilters?
\end{question*}
In the last part of the paper, we answer this question and generalize Theorem~\ref{theorem: fmz} to the setting of ideals on $P_\kappa(\lambda)$, using two-cardinals filter games where sets are played in analogy to the Welch games. This enables us to construct dense trees consisting of sets.
\begin{reptheorem}{theorem: game ideal}
    Suppose that $2^\lambda=\lambda^+$, there is no saturated ideal on $P_\kappa(\lambda)$, and that the Judge has a winning strategy for the game $G^\gamma_2$ for some regular $\omega<\gamma\leq\lambda$.
    Then there is an ideal $I$ on $P_\kappa(\lambda)$ such that:
    \begin{enumerate}
        \item $I$ is normal.
        \item $I$ is precipitous.
        \item $I^+$ has a dense subtree (ordered by $\supseteq_{\mathcal{F}}$) $T$ which is $\gamma$-closed.
        \item $I$ is $\lambda$-measuring. 
        That is, for any $\l A_\alpha\mid \alpha<\lambda\r\subseteq P_\kappa(\lambda)$ and any $S\in I^+$ there is $S'\subseteq S$, $S'\in I^+$ such that for any $\alpha<\lambda$, $S'\subseteq_{\mathcal{F}} A_\alpha$ or $S'\subseteq_{\mathcal{F}} P_\kappa(\lambda)\setminus A_\alpha$. 
    \end{enumerate}
\end{reptheorem}
The paper is organized as follows:
\begin{enumerate}
    \item Section \S~1: We begin with some preliminary definitions related to the Galvin property.
    We introduce other combinatorial properties of ultrafilters and show some basic relations between these properties, the Galvin property, and the Rudin-Keisler order.
    \item Section \S~2: Here we prove the main result of the paper and other results on the Galvin property on fine two-cardinal ultrafilters.
    We also prove some results about the Galvin property on ultrafilters on successor cardinals under the Ultrapower Axiom.
    \item Section \S~3: We show some applications of our results on Galvin's property from the previous section:
    We consider the Galvin property under the Ultrapower Axiom, density of ground model sets after the supercompact Prikry forcing, and generating sets of $P_\kappa(\lambda)$ measures.
    \item Section \S~4: We analyze filter combinatorics modulo the filter $Fine_{\kappa,\lambda}$:
    We consider a revised Galvin property with respect to this ideal, define a version of $P$-point filters, define a modification of the diagonal intersection, and we construct a two-cardinal analog of the ideal constructed in section $5$ of \cite{foremanmagidorzeman}.
\end{enumerate}
\section{Preliminaries}
\subsection{Ultrafilters and Ultrapowers}

Here we collect some definitions and basic facts about ultrafilters. Let $U$ be an ultrafilter over a set $X$.
We let $M_U$ denote the class of equivalence classes of functions $f$ with domain $X$ and let $j_U$ denote the usual ultrapower construction of the set-theoretic universe $V$. Namely, $j_U\colon V \to M_U$ is defined by $j_U(x)=[c_x]_U$, where $c_x$ is the constant function with value $x$.
We denote by $\id$ the identity function and we use $\id_U$ as shorthand for the class of the identity function $\id_U$. 
When $M_U$ is well-founded, we identify $M_U$ with its Mostowski collapse, and do the same with $\id_U$.
\begin{definition}
    Let $U,W$ be ultrafilters over $X,Y$ respectively.
    We say $U$ is \emph{Rudin-Keisler below $W$}, written $U \leq_{RK} W$, if there is $f:Y\to X$ such that 
    $U=f_*(W)$, where $$f_*(W)=\{A\subseteq X\mid f^{-1}X \in W\}.$$
    We say $U$ is \emph{Rudin-Keisler equivalent to $W$}, written $U \cong_{RK} W$ if $U \leq_{RK} W$ and $W \leq_{RK} U$.
\end{definition}
It is well known (see for example \cite{NegropontisComfort}) that $U\cong_{RK}W$ if and only if there is $f:Y\to X$ which is one-to-one and $f_*(W)=U$. 
The following facts about the Rudin-Keisler order are standard, see for example \cite{Goldberg+2022}.
\begin{proposition}
    Let $U$ and $W$ be ultrafilters.
    Then $U \leq_{RK} W$ if and only if there is an elementary $k \colon M_U \to M_W$ such that $k \circ j_U = j_W$.
    If $k$ is also a surjection, then also $U \cong_{RK} W$.
\end{proposition}
    Let $\kappa$ and $\lambda$ be cardinals.
    We set 
    \[P_\kappa(\lambda) = \{ x \subseteq \lambda : |x| < \kappa\}.\]
    The ultrafilters in this paper all have $P_\kappa(\lambda)$ as an underlying set. 
\begin{definition}
    Let $U$ be an ultrafilter on $P_\kappa(\lambda)$.
    We say $U$ is \emph{fine} if for all $\alpha <\lambda$ we have $\{x \in P_\kappa(\lambda) : \alpha \in x\} \in U$.
    We say $U$ is \emph{normal} if whenever $X_\alpha \in U$ for $\alpha < \lambda$ then the \emph{diagonal intersection} $\triangle_{\alpha < \lambda} X_\alpha \in U$ where
    \[ \triangle_{\alpha < \lambda} X_\alpha \coloneq \{x \in P_\kappa(\lambda) : \forall \alpha \in x,\,(x \in X_\alpha) \} \]
\end{definition}
\begin{proposition}\cite[p. 185]{Goldberg+2022}
    Let $U$ be an ultrafilter over $P_\kappa(\lambda)$. 
    \begin{enumerate}
        \item $U$ is fine if and only if $M_U\models j_U[\lambda]\subseteq \id_U$.
        \item $U$ is normal if and only if $M_U\models\id_U\subseteq j_U[\lambda]$. \footnote{Although $j_U[\lambda]$ might not be in $M_U$, we abuse notation by writing $M_U\models j_U[\lambda]\subseteq X$ to mean that for every $\alpha<\lambda$, $M_U\models j_U(\alpha)\in X$.}
    \end{enumerate}
\end{proposition}

Note that every fine ultrafilter $U$ on $P_\kappa(\lambda)$ is \textit{uniform}, namely, every set $X\in U$ has cardinality $\lambda$. 
Otherwise $|X|<\lambda$, and $|\bigcup X|\leq |X|\cdot\kappa<\lambda$.
Take any $\alpha\in \lambda\setminus \bigcup X$. 
If $U$ is fine, there is some $x_0\in\{y\in P_\kappa(\lambda): \alpha\in y\}\cap X\neq\emptyset$ but then $\alpha\in x_0\subseteq \bigcup X$, contradicting our choice of $\alpha$.

We say an ultrafilter $U$ is \emph{$\kappa$-complete} if whenever $X_\alpha \in U$ for all $\alpha<\kappa$ then $\bigcap_{\alpha<\kappa} X_\alpha \in U$. We say that $U$ is $\sigma$-complete if it is $\omega_1$-complete.

\begin{definition}
    A cardinal $\kappa$ is \textit{strongly compact} if for all $\lambda \geq \kappa$ there is a $\kappa$-complete fine ultrafilter on $P_\kappa(\lambda)$.
    $\kappa$ is \textit{supercompact} is for all $\lambda \geq \kappa$ there is a normal 
    , fine ultrafilter on $P_\kappa(\lambda)$.
\end{definition}
Note that any normal, fine ultrafilter on $P_\kappa(\lambda)$ is automatically $\kappa$-complete \cite{kanamoribook}.
\begin{definition}
    An ultrafilter $U$ is $(\kappa,\lambda)$-regular if there is a sequence of sets $\langle X_\alpha : \alpha < \lambda \rangle \subseteq U$ such that whenever $I \subseteq \lambda$ and $|I| = \kappa$, $\bigcap_{\alpha \in I} X_\alpha = \emptyset$.
\end{definition}
Regularity of ultrafilters can be viewed as a strengthening of the failure of Galvin's property, since any witness for $(\kappa,\lambda)$-regularity of an ultrafilter $U$ is also a witness for $\neg\Gal(U,\kappa,\lambda)$. Regularity admits a useful characterization in terms of the ultrapower:
\begin{proposition}
    Let $U$ be an ultrafilter.
    The following are equivalent:
    \begin{enumerate}
        \item $U$ is $(\kappa,\lambda)$-regular.
        \item There is a set $X \in M_U$ such that $M_U\models j_U[\lambda] \subseteq X$ and $M_U\models |X|<j_U(\kappa)$. 
    \end{enumerate}
\end{proposition}
Hence, regular and fine ultrafilters are connected in the following way:
\begin{corollary}\label{cor: regular to fine ultrafilter}
    The following are equivalent:
    \begin{enumerate}
        \item $U$ is $(\kappa,\lambda)$-regular.
        \item there is a fine ultrafilter on $P_\kappa(\lambda)$ such that $U\leq_{RK} W$ 
    \end{enumerate}
\end{corollary}

\section{The Galvin Property}
Given a filter $F$, recall that the Galvin propery (see Definition \ref{def: galvin property})  $\Gal(F,\kappa,\lambda)$ says that for all $\l A_i : i < \lambda \r\subseteq F$, there is $I \subseteq \lambda$ with $|I| = \kappa$ such that $\bigcap_{i \in I} A_i \in F$.
We will present the proof of Galvin's theorem, not only for completeness but also because our proof of Theorems \ref{theorem: cf(kappa) galvin} and \ref{thm: wierd Galvin} mirrors this one.
\begin{theorem}[Galvin's Theorem]\label{thm: galvin theorem}
 Suppose that $2^{<\kappa}=\kappa$ and let $F$ be a normal filter over $\kappa$. Then $\Gal(F,\kappa,\kappa^+)$. Namely, letting $\langle X_i\mid i<\kappa^+\rangle$ be a sequence of sets such that for every $i<\kappa^+$, $X_i\in F$, there is $Y\subseteq \kappa^+$ of cardinality $\kappa$, such that  $\bigcap_{i\in Y}X_i\in F$.
\end{theorem}
\begin{proof}
     For every $\alpha<\kappa^+$ and $\xi<\kappa$, let
$$H_{\alpha,\xi}=\{i<\kappa^+\mid X_i\cap\xi=X_\alpha\cap\xi\}.$$
\begin{claim}
There is $\alpha^*<\kappa^+$ such that for every $\xi<\kappa$, $|H_{\alpha^*,\xi}|=\kappa^+$ 
\end{claim}
\textit{Proof of claim.} Otherwise, for every $\alpha<\kappa^+$ there is $\xi_\alpha<\kappa$ such that $|H_{\alpha,\xi_\alpha}|\leq\kappa$.
By the pigeonhole principle,, there is $X\subseteq \kappa^+$
and $\xi^*<\kappa$, such that $|X|=\kappa^+$ and for each $\alpha\in X, \ \xi_\alpha=\xi^*.$
Since $\kappa$ is strong limit and $\xi^*<\kappa$, there are less than $\kappa$ many possibilities for $X_\alpha\cap \xi^*$. Hence we can shrink $X$ to $X'\subseteq X$ such that $|X'|=\kappa^+$ and find a single set $E^*\subseteq \xi^*$ such that for every $\alpha\in X'$, $X_\alpha\cap\xi^*=E^*$.
It follows that for every $\alpha\in X'$:
$$H_{\alpha,\xi_\alpha}=H_{\alpha,\xi^*}=\{i<\kappa^+\mid X_i\cap \xi^*=E^*\}.$$
Hence the set $H_{\alpha,\xi_\alpha}$ does not depend on $\alpha$, which means it is the same for every $\alpha\in X'$. Denote this set by $H^*$.
To see the contradiction, note that for every $\alpha\in X'$, $\alpha\in H_{\alpha,\xi_\alpha}=H^*$, thus $X'\subseteq H^*$, hence $$\kappa^+=|X'|\leq|H^*|\leq \kappa$$
contradiction.$\square_{\text{Claim}}$

Let $\alpha^*$ be as in the claim. Let us choose $Y\subseteq \kappa^+$ that witnesses the lemma. 
By recursion, define $\beta_i$ for $i<\kappa$. At each step we pick $\beta_i\in H_{\alpha^*,i+1}\setminus\{\beta_j\mid j<i\}$ (at the first step pick any index in $H_{\alpha^*,1}$). It is possible to find such $\beta_i$, since the cardinality of $H_{\alpha^*,\rho_i+1}$ is $\kappa^+$,  and $\{\beta_j\mid j<i\}$ is of size less than $\kappa$.
Let us prove that $Y=\{\beta_i\mid i<\kappa\}$ is as wanted. Indeed, by definition, it is clear that $|Y|=\kappa$. Next, we need to prove that $\bigcap_{\gamma\in Y}X_\gamma=\bigcap_{i<\kappa}X_{\beta_i}\in F$. By normality of $F$, $$X^*:= X_{\alpha^*}\cap\Delta_{i<\kappa}X_{\beta_i}\in F.$$ Thus it suffices to prove that $X^*\subseteq \bigcap_{i<\kappa}X_{\beta_i}$. Let $\zeta\in X^*$, then for every $i<\zeta$, $\zeta\in X_{\beta_i}$ by the definition of the diagonal intersection. For $i\geq\zeta$, $\zeta\in i+1$, and thus $\zeta<i+1$. Recall that $\beta_i\in H_{\alpha^*,i+1}$ which means that  $$X_{\alpha^*}\cap(i+1)=X_{\beta_i}\cap(i+1),$$ and since $\zeta\in X_{\alpha^*}\cap(i+1)$, $\zeta\in X_{\beta_i}$. We conclude that $\zeta\in\bigcap_{i<\kappa}X_{\beta_i}$, thus $X^*\subseteq\bigcap_{i<\kappa}X_{\beta_i}$. 
\end{proof}
The Galvin property is an invariant of the Rudin-Keisler order (or more generality the Tukey order, see \cite{BenhamouDobrinen}).
\begin{lemma}[{\cite[Lemma 2.1]{nonGalvinFilter}}]\label{lemma: galvin RK}
    Suppose $W \leq_{RK} U$ are ultrafilters and $\kappa\leq\lambda$ are cardinals.
    If $\lnot \Gal(W,\kappa,\lambda)$ then $\lnot \Gal(U,\kappa,\lambda)$.
\end{lemma}

The main question inspiring this paper concerns the Galvin property on ultrafilters on $P_\kappa(\kappa^+)$, but we shall answer the question more generally for ultrafilters on $P_\kappa(\lambda)$ for any $\lambda \geq \kappa^+$.
We shall split this section to three subsections, each addressing this general question depending on the cofinality of $\lambda$.

\subsection{The case $\cf(\lambda) > \kappa$}

We begin with a lemma that we shall use repeatedly.
\begin{lemma}\label{lemma: unbounded lemma}
    Let $\kappa < \lambda$ be cardinals and let $U$ be an ultrafilter on $P_\lambda(\lambda)$ such that $U$ does not concentrate on $P_\kappa(\zeta)$ for any $\zeta < \lambda$.
    Suppose that $[X \mapsto \mathcal{A}_X]_U = \mathcal{A} \in M_U$ and $\delta<\cf(\lambda)$. 
    If $M_U\models |\mathcal{A}| < j_U(\delta)$. Then for any $B \in U$ there is $\theta <\delta$ such that $
    \{ \sup(X) : X\in B \text{ and } |\mathcal{A}_X| < \theta \}$ is unbounded in $\lambda$.
\end{lemma}
\begin{proof}
    For every $\theta < \delta$ let $S_\theta = \{ \sup(X) : X \in B \text{ and } |\mathcal{A}_X| < \theta\}$.
    We want to show that there is some $\theta$ such that $\sup S_\theta = \lambda$.
    Since $M_U \models |\mathcal{A}| < j_U(\delta)$, by \L o\'{s} we have $Y \coloneq \{X \in B : |\mathcal{A}_X| < \delta\} \in U$.
    Note that $$Y = \bigcup_{\theta \leq \delta} \{X \in B : |\mathcal{A}_X| < \theta \}.$$
    Suppose now towards a contradiction that $\sup S_\theta < \lambda$ for all $\theta < \cf(\lambda)$. It follows that $$\zeta \coloneq \sup\{\sup(X) : X \in Y\}  < \lambda\ \ \text{   (recall }\cf(\lambda) > \delta).$$
    But then $U$ concentrates on $P_\kappa(\zeta)$, contradicting our assumption that $U$ is uniform.
\end{proof}
\begin{remark}
     Note that in the previous lemma, if we only assume $M_U\models |\mathcal{A}| \leq j_U(\delta)$, then we still get that for any $B \in U$ $
    \{ \sup(X) : X\in B \text{ and } |\mathcal{A}_X| \leq \delta \}$ is unbounded in $\lambda$. The reason is that for $\{X\in B\mid |\mathcal{A}_X|\leq\delta\}\in U$ and for any set in $Z\in U$,  $\{\sup(X)\mid X\in Z\}$ is unbounded in $\lambda$.
    
\end{remark}
Now we answer \cite[Question 6.4]{Benhamou_Goldberg_2024} of the first author and Goldberg.
\begin{theorem}\label{theorem: normal nongalvin}
    Let $U$ be a normal, fine ultrafilter on $P_\kappa(\lambda)$.
    Suppose that $2^{<\lambda}=\lambda$ and $\cf(\lambda) > \kappa$.
    Then $\lnot \Gal(U,\kappa,2^\lambda)$.
\end{theorem}
\begin{proof}
    Let $\mathcal{A} = \{j_U(Y) \cap (\sup j_U[\lambda]) :  Y \subseteq \lambda\}$.
    By Theorem $4.3.4$ in \cite{Goldberg+2022}, $\mathcal{A} \in M_U$.
    Let $f \colon P_\kappa(\lambda) \to P(P(\lambda))$ represent $\mathcal{A}$ in $M_U$, so that $j_U(f)(j_U[\lambda]) = \mathcal{A}$.
    For each $X \in P_\kappa(\lambda)$ let $\mathcal{A}_X = f(X)$.
    We may assume that for every $X$, $f(X) \subseteq P(\sup(X))$. For every $Y \subseteq \lambda$,  define
    \[ B_Y \coloneq \{X \in P_\kappa(\lambda) : Y \cap \sup(X) \in \mathcal{A}_X\},\]
    then $B_Y \in U$ as $j_U[\lambda]\in j_U(B_Y)$ by definition of $\mathcal{A}$.
    We claim that $\{B_Y : Y \subseteq \lambda\}$ witnesses $\lnot \Gal(U,\kappa,2^{\lambda})$.
    Otherwise, there is a sequence of distinct $\langle Y_i \subseteq \lambda: i < \kappa \rangle$ so that $B = \bigcap_{i<\kappa} B_{Y_i} \in U$.
    
    Note that since $\lambda<j_U(\kappa)$, and $j_U(\kappa)$ is measurable in $M_U$, $M_U\models |\mathcal{A}|=2^\lambda<j_U(\kappa)$ so by Lemma \ref{lemma: unbounded lemma}, we can fix $\theta_0<\kappa$ such that $S_{\theta_0}$ is unbounded in $\lambda$. 
    For $i\neq j < \theta_0^+$ let $\beta_{i,j}$ be least such that $Y_i \setminus \beta_{i,j} \neq Y_j \setminus \beta_{i,j}$.
    Since $\cf(\lambda)>\kappa>\theta_0$, $ \beta^*\coloneq\sup_{i\neq j<\theta_0^+}(\beta_{i,j})<\lambda$, and by our choice of $\theta_0$, there is an $X^* \in B$ with $|\mathcal{A}_{X^*}| \leq \theta_0$ and $\sup(X^*) \geq \beta^*$.
    Then the map $i \mapsto Y_i \cap \sup(X^*)$ is now a one-to-one map from $\theta_0^+$ into $\mathcal{A}_{X^*}$, a contradiction.
\end{proof}
\begin{corollary}
    If $2^\kappa = \kappa^+$ and $U$ is a normal, fine ultrafilter on $P_\kappa(\kappa^+)$ then $\lnot \Gal(U,\kappa,2^{\kappa^+})$.
\end{corollary}

Notice that the assumption $\cf(\lambda) > \kappa$ is crucial for the argument to work, so that the supremum of the $\beta_{i,j}$'s is below $\lambda$.
The case when $\cf(\lambda) < \kappa$ will admit a much easier argument.

Normality is crucial to show that the set $\mathcal{A}$ is in $M_U$ to produce the counterexample to the Galvin property.
Without normality, we cannot guarantee this, but we can effectively \textit{cover} $\mathcal{A}$ in the ultrapower if $U$ is fine.
This yields the following result:
\begin{theorem}\label{theorem: fine nongalvin}
    Let $U$ be a fine, $\sigma$-complete $P_\kappa(\lambda)$-ultrafilter where $\kappa$ is strong limit, $\cf(\lambda) > \kappa$, and $2^{<\lambda} = \lambda$.
    Then $\lnot \Gal(U,\kappa,2^{\lambda})$.
\end{theorem}
\begin{proof}
    For each $\alpha$ with $\kappa \leq \alpha < \lambda$, let $f_\alpha \colon \lambda \to P(\alpha) $ be a surjection (here we use $2^{<\lambda} = \lambda$.
    Since $U$ is fine, $\id_U \supseteq j_U[\lambda]$.
    It follows that $j_U(f_\alpha)[\id_U ]\supseteq j_U[P(\alpha)]$.

    Now let $\langle f^*_\alpha : \alpha < j_U(\lambda) \rangle = j_U(\langle f_\alpha : \alpha < \lambda \rangle),$ and set $X_\alpha = f^*_\alpha{} [\id_U].$
    Notice $\l X_\alpha : \alpha < j_U(\lambda) \r$ is in $M_U$: $X_\alpha \subseteq P(\alpha)^{M_U} \in M_U$ and $\id_U \in M_U$, and the sequence is definable from $\id_U $ and $\langle f^*_\alpha : \alpha < j_U(\lambda) \rangle$.
    Again, since $\id_U \supseteq j_U[\lambda]$ it follows that $j_U[P(\alpha)] \subseteq X_{j_U(\alpha)}$.
 
    Let $\lambda^* = \sup j_U[\lambda]$.
    In $M_U$, define the following set.
    \[\mathcal{A} = \{ Y \subseteq \lambda^* : Y \cap \alpha \in X_\alpha \text{ for unboundedly many } \alpha \in \id_U\cap \lambda^*\}\]

    Work in $M_U$, we have that $|\id_U|<j_U(\kappa)$ and $j_U(\kappa)$ is strong limit. Therefore, $$|\mathcal{A}| \leq |\prod_{\alpha\in \id_U\cap \lambda^*} X_\alpha|\leq |\id_U|^{|\id_U\cap\lambda^*|}< j_U(\kappa).$$
    By definition of $\mathcal{A}$ and elementarity we have $j_U(Y) \cap \lambda^* \in \mathcal{A}$ for all $Y \subseteq \lambda$.
    To see this, we will prove that for all $\alpha \in j_U[\lambda]$ we have $j_U(Y) \cap \lambda^* \cap \alpha \in X_\alpha$.
    Note that if $\alpha = j_U(\beta)$ for some $\beta < \lambda$ then $$j_U(Y) \cap j_U(\beta) = j_U(Y \cap \beta) \in j_U[P(\beta)] \subseteq X_\beta.$$

    Now let $f \colon P_\kappa(\lambda) \to P(P(\lambda^*))$ represent $\mathcal{A}$ in the ultrapower so that $$j_U(f)(\id_U) = \mathcal{A}.$$
    For each $X \in P_\kappa(\lambda)$ let $\mathcal{A}_X = f(X)$.
    Let $g$ represent $\lambda^*$ in $M_U$, so $j_U(g)(\id) = \lambda^*$.
    By the previous paragraph, $j_U(Y) \cap \sup(\id_U) \in \mathcal{A} = j_U(f)(\id)$.
    Reflecting this we have
    \[\forall Y \subseteq \lambda, \, B_Y \coloneq \{X \in P_\kappa(\lambda) : Y \cap g(X) \in \mathcal{A}_X\} \in U\]
    This is true as $\id_U \in j_U(B_Y)$ by elementarity and the definition of $\mathcal{A}$.
    We claim that $\{B_Y : Y \subseteq \lambda\}$ witnesses $\lnot \Gal(U,\kappa,2^{\lambda})$.
    Otherwise there is $\langle Y_i \subseteq \lambda: i < \kappa \rangle$ so that $B = \bigcap_{i<\kappa} B_{Y_i} \in U$.

    By Lemma \ref{lemma: unbounded lemma}, fix $\theta_0<\kappa$ such that $\{\sup(X) :  X\in B \text{ and } |\mathcal{A}_X| \leq \theta_0\}$ is unbounded in $\lambda$. Now for $i\neq j < \theta_0^+$ let $\beta_{i,j}$ be least such that $$Y_i \setminus \beta_{i,j} \neq Y_j \setminus \beta_{i,j}.$$
    Now take $\alpha^* > \sup_{i\neq j<\theta_0^+}(\beta_{i,j})$ below $\lambda$ such that there is an $X^* \in B$ with $|\mathcal{A}_{X^*}| \leq \theta_0$ and $g(X^*) = \alpha^*$.
    Then the map $i \mapsto Y_i \cap g(X^*)$ is a 1-1 map from $\theta_0^+$ into $\mathcal{A}_{X^*}$, a contradiction.
\end{proof}
Together with Corollary \ref{cor: regular to fine ultrafilter} we get the following, improving a result of Taylor \cite[Theorem 2.4 (2)]{TAYLOR197933} in some cases.
\begin{corollary}\label{cor: regular non-galvin}
    Let $\kappa$ be a strong limit cardinal. If $U$ is a $(\kappa,\lambda)$-regular and $\sigma$-complete ultrafilter over $\lambda$ where $\cf(\lambda)>\kappa$ and $2^{<\lambda}=\lambda$ then $\neg \Gal(U,\kappa,2^\lambda)$.
\end{corollary}
\begin{proof}
    By Corollary \ref{cor: regular to fine ultrafilter}, there is fine ultrafilter $W$ on $P_\kappa(\lambda)$ such that $W\leq_{RK} U$. Also, since $U$ is $\sigma$-complete so is $W$. By Theorem \ref{theorem: fine nongalvin}, $\neg \Gal(W,\kappa,2^\lambda)$. By Lemma \ref{lemma: galvin RK}, also $\neg \Gal(U,\kappa,2^\lambda)$.
\end{proof}
This corollary is also interesting in the light of the conjecture from \cite{TomJustinLuke} that non-Tukey-top ultrafilters (i.e. those which satisfy the Galvin property) must be non-regular.


\subsection{The case $\cf(\lambda) < \kappa$}

Here the situation trivializes due to a powerful result of Goldberg, which generalizes a classic result of Solovay.
We begin with some definitions.
\begin{definition}
    An ultrafilter $U$ on a cardinal $\lambda$ is \emph{weakly normal} if whenever $A \in U$ and $f \colon A \to \lambda$ is regressive, then there is a $B \subseteq A$ in $U$ such that $| f [A]| < \lambda$.
\end{definition}
Weakly normal ultrafilters admit a nice characterization in the ultrapower (see \cite[Proposition 4.4.23]{Goldberg+2022}).
\begin{proposition}
    Let $\lambda$ be a cardinal.
    A countably complete, uniform ultrafilter $U$ on $\lambda$ is weakly normal if and only if $\id_U$ is the unique ordinal $\alpha \geq \sup j_U[\lambda]$ such that $\alpha \neq j_U(f)(\beta)$ for any $f \colon \lambda \to \lambda$ and $\beta < \alpha$ (i.e. $\id_U$ is the unique generator of $j_U$ above $j_U[\lambda]$).
\end{proposition}
\begin{definition}
    Let $\theta$ be a cardinal.
    An ultrafilter $U$ on $\theta$ is \emph{isonormal} if $U$ is weakly normal and $M_U$ is closed under $\theta$-sequences.
\end{definition}
Goldberg gives an exact characterization of when an ultrafilter on a cardinal is equivalent to a supercompactness measure. 
Using this characterization, we can reduce the study of $P_\kappa(\lambda)$ ultrafilters to ultrafilters on cardinals when we are considering only Rudin-Keisler invariant properties, such as the Galvin property.
We shall use this deep fact again in the next section.
\begin{theorem}[Goldberg {\cite[Theorem 4.4.37]{Goldberg+2022}}]\label{Goldberg Theorem}
    Every normal, fine ultrafilter is Rudin-Keisler equivalent to a unique isonormal ultrafilter.
\end{theorem}
\begin{remark}\label{remark: isonormal theta}
    The proof of the above theorem also shows the following fact that we shall use in the next section: suppose $U$ a normal, fine ultrafilter on $P_\kappa(\lambda)$ and $\cf(\lambda) \geq \kappa$.
    Let $W$ is the unique isonormal ultrafilter isomorphic to $U$.
    Then there is some ordinal $\theta$ such that $\sup j_U[\lambda] \leq \theta < j_U(\lambda)$ and $X \in W$ iff $\theta \in j_U(X)$.
\end{remark}

This yields an even stronger failure of Galvin's property in this case.
\begin{theorem}
    Let $\cf(\lambda)<\kappa<\lambda$ and $2^{<\lambda} = \lambda$.
    If $U$ is a normal, fine ultrafilter on $P_\kappa(\lambda)$ then $\lnot \Gal(U,\kappa,2^{\lambda^+})$.
\end{theorem}
\begin{proof}
    Applying Theorem~\ref{Goldberg Theorem}, $U$ is Rudin-Keisler equivalent to an isonormal ultrafilter $W$ on $|\theta| \geq \lambda$, where $\theta$ is as in the previous remark.
    Note that $M_U$ and $M_W$ are in fact closed under $\lambda^{<\kappa}$-sequences and in this case $\lambda^{<\kappa} \geq \lambda^+$.
    Hence we may derive a normal fine ultrafilter $Z$ on $P_\kappa(\lambda^+)$ from $j_W$ using $j_W[\lambda^+]$.
    Now Theorem \ref{theorem: normal nongalvin} applies to $Z$ and $Z \leq_{RK} W \cong_{RK}U$, so by Lemma \ref{lemma: galvin RK} we may conclude $\lnot \Gal(U,\kappa,2^{\lambda^+})$.
\end{proof}

\subsection{The case $\cf(\kappa) = \lambda$}

Here the situation is more complicated, and the only one yielding a positive result. First we will show that the proof of Galvin's theorem for normal filters on a cardinal actually generalizes:

\begin{theorem}\label{theorem: cf(kappa) galvin}
    Suppose $\cf(\lambda) = \kappa$ and $2^{<\lambda} = \lambda$.
    Let $U$ be a normal, fine ultrafilter on $P_\kappa(\lambda)$.
    Then $ \Gal(U, \kappa, \lambda^+)$.
\end{theorem}
\begin{proof}
    By Theorem $4.4.37$ in \cite{Goldberg+2022}, there is a unique weakly normal ultrafilter $W$ on $\lambda$ such that $W \cong_{RK} U$. 
    Furthermore, $M_W$ is closed under $\lambda$-sequences.
    We will show $\Gal(W, \kappa, 2^\lambda)$ holds.
    Let $\l X_i : i < \lambda^+ \r \subseteq W$.
    For $\alpha < \lambda^+$ and $\xi < \lambda$ let
    \[H_{\alpha, \xi} = \{ i < \lambda^+ : X_i \cap \xi = X_\alpha \cap \xi \}\]
    \begin{claim}
        There is an $\alpha^* < \lambda^+$ such that $|H_{\alpha^*,\xi}| = \lambda^+$ for every $\xi < \lambda$.
    \end{claim}
    \begin{claimproof}
    Otherwise, for every $\alpha < \lambda^+$ there is $\xi_\alpha < \lambda$ such that $|H_{\alpha,\xi_\alpha}| \leq \lambda$.
        There must be $X \subseteq \lambda^+$ with $|X| = \lambda^+$ and a $\xi^* < \lambda$ such that $\xi_\alpha = \xi^*$ for all $\alpha \in X$.
        Since $2^{<\lambda}=\lambda$ and $\xi^* < \lambda$ there are fewer than $\lambda$ many possibilities for $X_\alpha \cap \xi^*$.
        Hence we may shrink $X$ to $X' \subseteq X$ with $|X'| = \lambda^+$ and find a $E^* \subseteq \xi^*$ such that $X_\alpha \cap \xi^* = E^*$ for every $\alpha \in X'$.
        Then for every $\alpha \in X'$, $H_{\alpha,\xi_\alpha} = H_{\alpha,\xi^*}$.
        Denote this set by $H^*$.
        We see that $X' \subseteq H^*$, but $|X'|$ is strictly larger than $|H^*|$, a contradiction.
    \end{claimproof}
    
    Let $\alpha^*$ be as in the claim.
    Fix $f : \lambda \to \kappa$ so that $[f]_W = \kappa$.
    Note that we may take $f$ to be monotone. 
    To see this, let $g \colon \kappa \to \lambda$ be increasing, continuous, and cofinal, and assume the range of $g$ consists only of cardinals.
    For every $\beta < \lambda$ let $f(\beta) = \alpha$ where $\alpha$ is least such that $g(\alpha)> \beta$.
    Then $f$ is clearly weakly monotone.
    Now we check that $[f]_W = \kappa$.
    As mentioned in Remark \ref{remark: isonormal theta}, there is some ordinal $\theta$ strictly between $\sup j_U[\lambda]$ and $j(\lambda)$ such that
    \[ X \in W \iff \theta \in j_U(X)\]
    So we shall show that $j_W(f)([\id]_W) = j_U(f)(\theta) = \kappa$.
    If $\alpha < \kappa$ then $j_U(g)(\alpha) = j_U(g(\alpha)) < \sup j_U[\lambda] < \theta$.
    Since $g$ is continuous, $\sup j_U[\lambda] = j_U(g)(\kappa)$.
    Since $g(\alpha)$ is always a cardinal, $\theta < |\sup j_U[\lambda]|^{+M_U} \leq j(g)(\kappa+1)$.
    Hence $j_U(f)(\theta) = \kappa$, as desired.

    Now for all $i < \kappa$ choose $\beta_i \in H_{\alpha^*,g(i)+1} \setminus \{\beta_j \mid j < i\}$ for $i < \kappa$.
    We define the $f$-diagonal intersection:
    \[D = \triangle^f_{i < \kappa} X_{\beta_i} = \{ \eta < \lambda \mid \forall i < f(\eta) \, (\eta \in X_{\beta_i})\}\]
    Note that $j_W(\vec{X})_i= j_W(X_{\beta_i})$ for all $i < \kappa$ and $j_W(f)([\id]_W) = \kappa$, so it follows that $W$ is closed under $f$-diagonal intersections.
    Now it suffices to show that $$X_{\alpha^*} \cap D \subseteq \bigcap_{i < \kappa} X_{\beta_i}.$$
    Towards this end suppose $\zeta \in X_{\alpha^*} \cap D$.
    Fix $i < \kappa$.
    If $i < f(\zeta)$ then $\zeta \in X_{\beta_i}$ by definition of $D$.
    If $i \geq f(\zeta)$, since $\zeta \in X_{\alpha^*}$ then $X_{\beta_i} \cap \lambda_i+1 = X_{\alpha^*} \cap \lambda_i+1$, so $\zeta \in X_{\beta_i}$.
\end{proof}

We cannot use the strategy in the proof of Theorem \ref{theorem: normal nongalvin} since the cofinality of $\lambda$ is too small in this case to show the failure of Galvin's property in the same way.
However we can use a simple counting argument to show some failure of Galvin's property.
The following proof was shown to the second author by Goldberg.
\begin{theorem}\label{theorem: cf kappa nongalvin}
    Suppose $\cf(\lambda) = \kappa$ and $2^{<\lambda}=\lambda$.
    Let $U$ be a normal, fine ultrafilter on $P_\kappa(\lambda)$.
    Then $\neg \Gal(U,(2^{\kappa})^+,2^\lambda)$.
\end{theorem}
\begin{proof}
    Let $W$ be the isonormal ultrafilter on $\lambda$ that is RK-equivalent to $U$.
    Let $\lambda^*\coloneq\sup(j_U[\lambda])$ and $\mathcal{A} = \l j_U(S) \cap \lambda^* \mid S \subseteq \lambda \r$. By Theorem \ref{Goldberg Theorem}, $\mathcal{A}\in M_W$ and let $\mathcal{A} = [\beta \mapsto \mathcal{A}_\beta]_W$. We may assume that for every $\beta<\lambda$, $|\mathcal{A}_\beta|<\kappa$.
    Let $\l X_\alpha \mid \alpha < 2^\lambda \r $ be a one-to-one enumeration of subsets of $\lambda$.
    For $\alpha < \lambda$ let
    \[ B_\alpha =  \{ \beta < \lambda : X_\alpha \cap \beta \in \mathcal{A}_\beta\}\]
    Suppose towards a contradiction that there is some $I \subseteq 2^\lambda$ such that $|I| = (2^\kappa)^+$ and $B \coloneq \bigcap_{\alpha \in I} B_\alpha\in U$. Let $C \subseteq B$ be cofinal in $\lambda$ with $ot(C) = \kappa$.
    Now consider
    \begin{align*}
        f \colon \,& I \to \prod_{\beta \in C} \mathcal{A}_\beta\\
        &\alpha \mapsto \l X_\alpha \cap \beta \mid \beta \in C \r
    \end{align*}
    Then $f$ is a well-defined injection by the definition of $B$ but $|I| = (2^\kappa)^+$ and the above product has size $2^\kappa$, a contradiction.
\end{proof}
The above proof shows that the sets $\l B_\alpha\mid \alpha<2^\lambda\r$ have the following property:
For any $I\in [2^\lambda]^{(2^\kappa)^+}$, $\bigcap_{\alpha\in I} B_\alpha$ must be bounded in $\lambda$.

\subsection{Possible strengthenings}
Under the assumptions that $2^{<\lambda}=\lambda$, the Galvin property for a normal fine ultrafilter on $P_\kappa(\lambda)$ is fully settled when $\cf(\lambda)\neq \kappa$. For $\cf(\lambda) = \kappa$, there is still one last case:
\begin{question}\label{question: cf kappa galvin}
    Let $U$ be a normal fine ultrafilter on $P_\kappa(\lambda)$, where $\cf(\lambda) = \kappa$.
    Must $\Gal(U,2^\kappa,2^\lambda)$ hold? Must $\Gal(U,2^\kappa,2^\lambda)$ fail?
\end{question}
We conjecture that the answer is independent.
Another question we would like to address is the necessity of the assumptions
of Theorem \ref{theorem: fine nongalvin} (or rather its corollary \ref{cor: regular non-galvin}) regarding $U$ being a $(\kappa,\lambda)$-regular $\sigma$-complete ultrafilter on $P_\kappa(\lambda)$. The cardinal arithmetic assumptions of this theorem will be considered in Section~\ref{Sec: small generating}.

Let us consider the fineness assumption. There is a simple counter example if we take $\lambda$ to be much larger than $\kappa$.
\begin{example}
    Suppose $\kappa<\lambda$ are measurable cardinals, and let $U_0,U_1$ be normal ultrafilters on $\kappa$ and $\lambda$ respectively. Let $W$ be an ultrafilter over $\lambda$ which is Rudin-Keisler equivalent to the product ultrafilter $U_0\cdot U_1$ on $\kappa\times \lambda$. So $W$ is a $\kappa$-complete ultrafilter on $\lambda$ and $2^{<\lambda}=\lambda$ but $W$ satisfies $\Gal(U,\lambda,\lambda^+)$. Indeed, if $\l A_\alpha\mid \alpha<\lambda^+\rangle\subseteq U$, then for each $\alpha$ we can find $A_{\alpha,0}\in U_0$ and $A_{\alpha,1}\in U_1$ such that $A_{\alpha,0}\times A_{\alpha,1}\subseteq A_\alpha$. We can now stabilize $A_{\alpha,0}$ for many $\alpha$'s, so we assume without loss of generality that $A_{\alpha,0}=X_0\in U_0$ for every $\alpha<\lambda^+$. Now we apply Galvin's theorem for the normal measure $U_1$ on $\lambda$, to find $I\in [\lambda^+]^{\lambda}$ such that $X_1=\bigcap_{i\in I}A_{\alpha,1}\in U_1$. It follows that $X_0\times X_1\subseteq \bigcap_{\alpha\in I}A_\alpha$, as wanted. 
\end{example}
   For ultrafilters over successor cardinals we need to work a bit harder. We can use recent results  \cite{TomJustinLuke} regarding the consistency of non-Tukey-top ultrafilters on successor cardinals.
For the convenience of the reader, we translate the results to our settings.
\begin{theorem}[Essentially {\cite[Thm. 5.4]{TomJustinLuke}}]
    Let $\nu^{<\nu}=\nu$. Suppose that $2^\nu<2^{\nu^+}$ and that there is a normal $\nu^+$-dense ideal\footnote{An ideal $I$ on $\rho$ is called $\mu$-dense if $P(\rho)/I$ has a dense subset of size $\mu$.} on $\nu^+$. Then there is a uniform ultrafilter $U$ on $\nu^+$ such that $\Gal(U,\nu^+,2^{\nu^+})$. 
\end{theorem}
Note that the ultrafilter produced by the above theorem is non-$(\nu,\nu^+)$-regular.
 The consistency of dense ideals on small cardinals was proven to be consistent by Woodin \cite{Woodin+2010}, and recent results of Eskew and Hayut \cite{eskew2024denseideals} made significant progress in the study of such ideals.
\begin{theorem}[\cite{eskew2024denseideals}]
    It is consistent (relative to large cardinals) that $\GCH$ holds and for every $n < \omega$ there is a uniform, normal $\omega_{n+1}$-dense ideal on $\omega_{n+1}$.
\end{theorem}
This gives the consistency of an ultrafilter on a successor cardinal which satisfies the Galvin property; however, it is unclear if this could be made the successor of a supercompact cardinal:
\begin{question}
    Is it consistent to have a normal $\kappa^+$-dense ideal on $\kappa^+$ where $\kappa$ is $\kappa^+$-supercompact?
\end{question}

\section{Some Applications}
\subsection{The Ultrapower Axiom} Here we present some results on the Galvin property under the Ultrapower Axiom.
First a concept that will be vital to the proof of the next theorem.
\begin{definition}
    A nonprincipal $\sigma$-complete ultrafilter $U$ is \emph{irreducible} if whenever there is an ultrafilter $W$ and $W' \in M_W$ such that $(M_{W'})^{M_W} = M_U$, either $W$ is principal or $W \leq_{RK} U$.
\end{definition}
\begin{theorem}
    (UA) Let $U$ be a $\sigma$-complete ultrafilter on $\kappa^+$.
    Then there is a fine, $\sigma$-complete ultrafilter $F$ on $P_\kappa(\kappa^+)$ such that $F \leq_{RK} U$.
\end{theorem}
\begin{proof}
    Let $U$ be a $\sigma$-complete ultrafilter on $\kappa^+$.
    By Theorem $8.2.24$ in \cite{Goldberg+2022} there is an ordinal $\lambda < \kappa^+$ and an ultrafilter $D$ on $\lambda$ with $D \leq_{RK} U$ such that $j_D(U) = W$ is a $j_D(\kappa^+)$-irreducible ultrafilter and $j_U = j_W^{M_D} \circ j_D$.
    By Theorem $8.2.22$ in \cite{Goldberg+2022} $j_W^{M_D}$ is a $j_D(\kappa^+)$ supercompactness embedding from the perspective of $M_D$.
    Hence $j_W^{M_D}[j_D(\kappa^+)] \in M_U$.
    Now define an ultrafilter $F$ on $P_\kappa(\kappa^+)$ by
    \[X \in F \iff j_W^{M_D}[j_D(\kappa^+)] \in j_U(X)\]
    then $F$ is $\sigma$-complete.
    To see that $F$ is fine, notice $j_D(\kappa^+) > \kappa^+$ we have $j_W^{M_D}[j_D(\kappa^+)] \supseteq j_U[\kappa^+]$.
    Also $j_U(\kappa) = j_W(j_D(\kappa)) > j_D(\kappa^+)$ since $j_W(\kappa) > \kappa^+$.
    Hence $$|j_W^{M_D}[j_D(\kappa^+)]|^{M_U} = |j_D(\kappa^+)|^{M_U} < j_U(\kappa).$$

    Since $j_W^{M_D}[j_D(\kappa^+)]$ covers $j_U[\kappa^+]$ and has size less than $j_U(\kappa)$ we conclude that $F$ is fine.
    Furthermore $F \leq_{RK} U$ since $F$ was derived using $j_U$.
\end{proof}
This yields an improvement of \cite[Cor. 6.2]{Benhamou_Goldberg_2024}.
\begin{corollary}\label{Cor: UA}
    (UA) Let $U$ be a $\sigma$-complete ultrafilter on $\kappa^+$.
    If $2^\kappa=\kappa^+$ then $\lnot \Gal(U,\kappa,2^{\kappa^+})$.
\end{corollary}
\subsection{Density of Old Sets} The failure of the Galvin property for supercompactness measures allows us to generalize a result of Gitik \cite{gitikdensity} and Benhamou-Garti-Poveda \cite{tomshimonalejandro} on ground model sets after Prikry forcing over a measurable cardinal.
First we start by recalling the definition of the supercompact Prikry forcing (see \cite{gitikhandbook}).

\begin{definition}
    Let $x,y \in P_\kappa(\lambda)$ we say $x$ is \emph{strongly below} $y$, written $x \prec y$, iff $x \subseteq y$ and $|x| < |\kappa \cap y|$.
\end{definition}

\begin{definition}\label{def:supecompact Prikry}
    Let $U$ be a normal ultrafilter on $P_\kappa(\lambda)$.
    The \emph{supercompact Prikry forcing} consists of $\langle \vec{x},A\rangle$ where $\vec{x}$ is a finite $\prec$-increasing sequence of elements in $P_\kappa(\lambda)$ and $A \in U$ is such that $\max_{\prec}(\vec{x}) \prec a$ for all $a \in A$.
    We set $\langle \vec{x}, A\rangle \leq \l \vec{y}, B \r$ if the following hold:
    \begin{enumerate}
        \item $\vec{x}$ is an end-extension of $\vec{y}$.
        \item $A \subseteq B$.
        \item Whenever $|\vec{y}| < i \leq |\vec{x}|$, $x_i \in B \setminus A$.
    \end{enumerate}
\end{definition}

\begin{theorem}\label{theorem: density old sets}
    Assume $2^{<\lambda} = \lambda$ and $\cf(\lambda) \neq \kappa$.
    Suppose $U$ is a normal $P_\kappa(\lambda)$ ultrafilter, and let $\mathbb{P}$ be the supercompact Prikry forcing with respect to $U$.
    In $V[G]$, there is a set $S \subseteq 2^\lambda$ that contains no ground model set of size $\kappa$.
\end{theorem}
\begin{proof}
    By Theorem \ref{theorem: normal nongalvin}, in $V$ there is $\l A_\alpha : \alpha < 2^\lambda \r \subseteq U$ such that no intersection of $\kappa$ many $A_\alpha$'s lies in $U$. 
    An easy density argument ensures that for every $\alpha<2^\lambda$ there is $n_\alpha<\omega$ such that for every $m>n_\alpha$, $X_m\in A_\alpha$. 
    Since $\cf(2^\lambda)>\omega$, there is $S\in [2^\lambda]^{2^\lambda}$, and $n^*<\omega$ such that for every $\alpha\in S$, $n_\alpha=n^*$. 
    This means that $\{X_m\mid m<n^*\}\subseteq\bigcap_{\alpha\in S}A_\alpha$.
    We will show $S$ contains no ground model set of size $\kappa$.
    Otherwise $T \in [S]^\kappa$ in $V$, then by the failure of Galvin's property, $\bigcap_{\alpha \in T} A_\alpha \notin U$. On the other hand,
    $\{X_m \mid m>n^*\}\subseteq \bigcap_{\alpha\in T} A_\alpha$. However, since $\kappa\setminus(\bigcap_{\alpha \in T} A_\alpha) \in U$, the same density argument as before yield that for some $n'<\omega$, for every $m>n'$, $X_m\notin \bigcap_{\alpha \in T} A_\alpha$, this is a contradiction. 
\end{proof}

We will give some example conclusions we may draw from this theorem.
Let $G$ be generic for the supercompact Prikry forcing, and let $X \in V[G] \setminus V$ such that $X \subseteq V[G]$.
Let $\eta = |X|$.
Then:
\begin{enumerate}
    \item If $C = \bigcup \{ \vec{x} : \langle \vec{x},A\rangle \in G \}$ is the Prikry sequence, then $C$ contains no infinite ground model set.
    \item If $\theta$ is any cardinal of $V[G]$ of countable cofinality, then $\theta$ has a subset of size $\theta$ with no ground model subset of the same cardinality.
    \item If $\omega < \cf^V(\eta) < \kappa$, then $X$ contains a ground model set of size $\eta$.
    \item If $\kappa < \cf(\eta) \leq 2^\lambda$ then $X$ contains ground model subsets of every cardinality $< \kappa$, but no ground model set of size $\kappa$.
    \item If $\cf(\eta) > 2^\lambda$ then $X$ contains a ground model set of the same size.
\end{enumerate}

\begin{remark}
    If $U$ is a supercompactness measure on $P_\kappa(\lambda)$ then there is a Rudin-Keisler projection from $U$ to a normal measure on $\kappa$, via $x \mapsto |x \cap \kappa|$.
This can be lifted to a projection from the supercompact Prikry forcing with $U$ to the classical Prikry forcing with a normal measure on $\kappa$.
Hence any new sets with no large ground model subsets added by the classical Prikry forcing will also persist into the supercompact Prikry extension.
Hence, Gitik's analysis of the density of old sets in a usual Prikry extension can be used. In fact, if $U$ is $2^\kappa$-supercompact, then any $\kappa$-complete ultrafilter on $\kappa$ is a Rudin-Keisler projection of $U$ and the Rudin-Keisler projection induces a projection of the corresponding forcings.
\end{remark}
\subsection{Generating Sets of $P_\kappa(\lambda)$-Measures}\label{Sec: small generating}

Recall that given an ultrafilter $U$ on a set $X$, and a filter $F$, we say that $\mathcal{B}\subseteq U$ is an $\subseteq_F$-base for $U$ if for every $A\in U$ there is $B\in\mathcal{B}$ such that $B \subseteq A \ (\text{mod} \ F)$, namely, there is $C\in F$ such that $B\cap C\subseteq A$. 
When $U$ is an ultrafilter over a cardinal $\mu$, we usually consider $F=\{\kappa\}$ or $F$ being the co-bounded filter on $\kappa$, in which case $ \subseteq  \ (\text{mod} \ F)$ translates to $\subseteq$ and $\subseteq^*$ respectively. 
However, on $P_\kappa(\lambda)$-measures, there is another natural ideal to consider-- the fine ideal. 
Let 
\[\chi_F(U)=\min\{|\mathcal{B}|\mid \mathcal{B}\text{ is a }\subseteq_F\text{-base for }U\}\]
The following is well-known (see for example \cite[Claim 1.2]{gartishelahsingular}):
\begin{lemma}\label{Lemma: size of base}
    No uniform ultrafilter over $X$ has a $\subseteq$-base consisting of $|X|$-many sets.
\end{lemma}
An immediate corollary is:
\begin{corollary}\label{cor: generated modulo F}
    Let $U$ be a uniform ultrafilter on $X$ and $F$ a filter over $X$. Suppose that $F\subseteq U$, and $F$ is generated by $|X|$-many sets, then $\chi_F(U)=\chi(U)$.  
\end{corollary}
The way our results connect to the above is by using the following lemma:
\begin{lemma}[Folklore]\label{lemma: galvin and base}
    Let $\lambda$ be a regular cardinal and $U$ an ultrafilter on $\lambda$.
    Suppose that $U$ satisfy $\neg \Gal(U,\lambda,\lambda)$.
    Then $\chi(U)\geq\lambda$.
\end{lemma}
\begin{proof}
    Let $\l A_\alpha\mid \alpha<\lambda\r\subseteq U$ witness that $\neg \Gal(U,\lambda,\lambda)$. Towards a contradiction, if $\chi(U)<\lambda$, let $\l X_\beta\mid \beta<\theta\r\subseteq U$ be a $\subseteq$-base for $U$ for some $\theta<\lambda$. For each $\alpha<\lambda$, we find $\beta_\alpha<\theta$ such that $X_{\beta_\alpha}\subseteq A_\alpha$. By the regularity of $\lambda$, we can find a single $\beta^*$ ans $S\in [\lambda]^\lambda$ such that for every $\alpha\in S$, $\beta_\alpha=\beta^*$. It follows that $X_{\beta^*}\subseteq \bigcap_{\alpha\in S}A_{\alpha}$, contradicting that $\l A_\alpha\mid \alpha<\lambda\r$ witnesses $\neg \Gal(U,\lambda,\lambda)$.
\end{proof}
\begin{corollary}\label{Cor: large ultrafilter number}
    If $U$ is a $\sigma$-complete $(\kappa,\lambda)$-regular ultrafilter over $\lambda$, $\kappa$ being a strong limit and $2^{<\lambda}=\lambda$, then $\chi(U)=2^\lambda$.
\end{corollary}
Next, let us show that $2^{<\lambda}=\lambda$ cannot be dropped from Theorem \ref{theorem: fine nongalvin}. Towards that, we need the following theorem:
\begin{theorem}[Raghavan-Shelah {\cite{Raghavan-Shelah}}]\label{Thm: RaghavanShelah}
    Suppose that $\kappa^{<\kappa}=\kappa$, and $\kappa<cf(\mu)<\lambda<\mu$, where $\mu$ is a strong limit cardinal and suppose $U$ is a $\cf(\mu)$-indecomposable ultrafilter over $\lambda$. 
    If $G \subseteq \add(\kappa,\mu)$ is a $V$-generic filter, any extension of $U$ to a $V[G]$-ultrafilter $U^*$ is generated by $\mu$-many sets. 
\end{theorem}
\begin{corollary}
    Let $\kappa$ be a Laver-indestructible supercompact cardinal
    \footnote{That is, after any $\kappa$-directed forcing, $\kappa$ remains supercompact \cite{Laver}.} 
    and let $\lambda>\kappa$ be a regular cardinal such that there is $\lambda$-complete fine ultrafilter $U$ over $P_\lambda(\lambda^+)$. 
    Then in $V[G]$, there is a $\kappa$-complete fine ultrafilter $U^*$ on $P_\lambda(\lambda^+)$ such that \\
    $\Gal(U^*,2^{\lambda^+},2^{\lambda^+})$. 
    In particular $\Gal(U^*,\kappa,2^{\lambda^+})$.
\end{corollary}
\begin{proof}
     Let $\bar{U}$ be the filter generated by $U$ in $V[G]$. 
     Since $\text{Add}(\kappa,\lambda^{+\kappa^+})$ is $\kappa$-closed, $\bar{U}$ is a $\kappa$-complete filter, and since $\kappa$ is indestructible, in $V[G]$, we can extend $\bar{U}$ to a $\kappa$-complete ultrafilter $U^*$ on $P_\lambda(\lambda^+)$. 
     To see that $U^*$ remains fine, it suffices to use the fineness of $U$ and note that $\{A\in P_\lambda(\lambda^+)\mid \alpha\in A\}^V\subseteq \{A\in P_\lambda(\lambda^+)\mid \alpha\in A\}^{V[G]}$. 
     It remains to see that $\Gal(U^*,2^{\lambda^+},2^{\lambda^+})$ holds, since $U$ is $\lambda$-complete, it is in particular $\kappa^+$-complete, so we may apply Theorem \ref{Thm: RaghavanShelah}, to conclude that the extension $U^*$ of $U$ is generated by $\lambda^{+\kappa^+}$-many sets. 
     Also note that in $V[G]$, $2^{\lambda^+}=\lambda^{+\kappa^++1}$. 
     By Lemma \ref{lemma: galvin and base}, we conclude that $\Gal(U^*,2^{\lambda^+},2^{\lambda^+})$ holds, as desired.
\end{proof}

\section{On generating sets modulo the fine filters}

In this section we discuss two results concerning generating sets of a fine filter with respect to $\subseteq_{\Fine(\kappa,\lambda)}$, where $\Fine(\kappa,\lambda)$ is the fine filter defined as follows:
\begin{definition}
    The \textit{fine filter} $\Fine(\kappa,\lambda)$ is the filter generated by sets for the form $\{X\in P_\kappa(\lambda)\mid i\in X\}$ for some $i<\lambda$. 
    For $X,Y \in P_\kappa(\lambda)$ we say $X \subseteq_{\Fine(\kappa,\lambda)} Y$ if there is some $A \in \Fine(\kappa,\lambda)$ such that $X \cap A \subseteq Y$.
    When $\kappa$ and $\lambda$ are clear from context, we will simply write $X \subseteq_\mathcal{F} Y$.
\end{definition}
Note that by Corollary \ref{cor: generated modulo F}, we have that $\chi_{\Fine(\kappa,\lambda)}(U)=\chi(U)$.
\subsection{The revised Galvin property}
    One may argue that the problem of generalizing Galvin's theorem to normal $P_\kappa(\lambda)$-ultrafilters is that the wrong version of Galvin's property was used. In this section, we present two possible ways the theory generalizes if one uses the inclusion modulo the fine ideal.
    \begin{definition}
    Let $U$ be a fine ultrafilter on $P_\kappa(\lambda)$.
    $\rGal(U,\mu,\mu')$ is the statement that from any $\mu'$-many sets in $U$ there are $\mu$-many which have a lower bound in $\subseteq_\mathcal{F}$.
    In other words, given $\l A_\alpha : \alpha < \mu' \r \subseteq U$ there is some $I \subseteq \mu'$ with $|I| = \mu$ and $A^* \in U$ such that whenever $\alpha \in I$, $A^* \subseteq_\mathcal{F} A_\alpha$.
    \end{definition}
\begin{definition}
    Let $F$ be a filter on a set $D$.
    Suppose $\mathcal{A}$ is a family of subsets of $D$ such that whenever $A_0,\ldots, A_{n-1} \in \mathcal{A}$, $\bigcap_{i < n}A_i \in F^+$.
    Then we denote by $F[\mathcal{A}]$ the minimal filter such that $F \cup \mathcal{A} \subseteq F[\mathcal{A}]$.
\end{definition}
There are limitations on the variation of Galvin's property that can hold. The following limitation is a slight modification of \cite[Thm. 4.3]{Tomcohesive}:
\begin{corollary}
    Let $U$ be a fine ultrafilter on $P_\kappa(\lambda)$ and let $\mu = \cf(\chi(U))$.
    Then $\lnot \rGal(U,\mu,\mu)$.
    In particular, if $2^\lambda = \lambda^+$ then $\lnot \rGal(U,\lambda^+,\lambda^+)$.
\end{corollary}
\begin{proof}
    Let $\l X_i : i < \chi_\mathcal{F}(U) \r$ be a $\subseteq_\mathcal{F}$-base for $U$.
    By Lemma \ref{Lemma: size of base}, $\chi_\mathcal{F}(U)=\chi(U) \geq \lambda^+$.
    By thinning out the $\subseteq_\mathcal{F}$-base, we may assume that for any $j < \chi(U)$, $X_j\notin \Fine(\kappa,\lambda)[\l X_i : i < j\r]$ , namely $X_i\not\subseteq_\mathcal{F} X_j$ for any $i<j$. 
    Now let $\l \alpha_i\mid i<\cf(\chi(U))\r$ be increasing and cofinal in $\chi(U)$. 
    We claim that $\l X_{\alpha_i}\mid i<\cf(\chi(U))\r$ witnesses that $\lnot \rGal(U,\lambda,\lambda)$. 
    Otherwise, there is $I$ unbounded in $\cf(\chi(U))$ such that $\{X_{\alpha_i}\mid i\in I\}$ has a $\subseteq_\mathcal{F}$-lower bound $X\in U$. 
    Then there is $j< \cf(\chi(U))$ such that $X_j\subseteq_\mathcal{F} X$. 
    Since $\l \alpha_i\mid i<\cf(\chi(U))\r$ is cofinal and $I$ is unbounded, there is $i\in I$ such that $j<\alpha_i$. 
    But this is impossible since this would mean that $X_j\subseteq_\mathcal{F} X\subseteq_\mathcal{F} X_{\alpha_i}$, contradicting the choice that $X_{\alpha_i}\notin \Fine(\kappa,\lambda)[\l X_j\mid j<\alpha_i\r]$. 
\end{proof}
\begin{definition}
    We say that a fine ultrafilter $U$ on $P_\kappa(\lambda)$ is a $P_\mu$-point if every collection $\l X_i\mid i<\rho\r\subseteq U$ such that $\rho<\mu$ there is a set $A\in U$ such that for every $i<\rho$, $A\setminus X_i\in \Fine(\kappa,\lambda)^*$, where $\Fine(\kappa,\lambda)^*$ is the ideal dual to $\Fine(\kappa,\lambda)$.
\end{definition}
For successor cardinals, Ketonen \cite{ketonenBenda} has another definition of a $p$-point which differs from this one.  
\begin{proposition}
    If $U$ is a normal fine ultrafilter then $U$ is a $P_{\lambda^+}$-point.
\end{proposition}
\begin{proof}
    Let $\rho < \lambda^+$ and let $\l X_i : i <\rho\r \subseteq U$. Let $\l X_{i_\nu}\mid \nu<\lambda\r$ be a re-enumeration of order-type $\lambda$.
    Let $$A = \triangle_{\nu < \lambda}X_{i_\nu} = \{x \in P_{\kappa}(\lambda) : x \in \bigcap_{\nu\in x}X_{i_\nu}\}.$$
    Then $A \in U$ by normality.
    Furthermore, for each $i<\rho$ we find $\nu<\lambda$ such that $i=i_\nu$. Then, $$A \setminus X_i=A\setminus X_{i_\nu} \subseteq \{x : i_\nu \notin x\} \in \Fine(\kappa,\lambda)^*.$$
\end{proof}
It is easy to see that the following holds:
\begin{proposition}
    If $U$ is a $P_{\lambda^+}$-point then $\rGal(U,\lambda,\lambda)$.
\end{proposition}
Let us provide some analogies to the usual characterization of $p$-points using partitions:
\begin{proposition}
    Let $U$ be a fine ultrafilter on $P_\kappa(\lambda)$.
    The following are equivalent:
    \begin{enumerate}
        \item $U$ is a $P_{\lambda^+}$-point.
        \item Whenever $\l X_i : i < \lambda \r \subseteq P(P_\kappa(\lambda))$ either there is $j < \lambda$ such that $X_j \in U$ or else there is an $X \in U$ such that $X \cap X_i \in \Fine(\kappa,\lambda)$ for all $i < \lambda$.
    \end{enumerate}
\end{proposition}
\begin{proof}
    Assume $U$ is a $P_{\lambda^+}$-point and
    let $\l X_i\mid i<\lambda\r\subseteq P(P_\kappa(\lambda)$ be such that for every $i<\lambda$, $X_i\notin U$. Then $\l X_i^c\mid i<\lambda\r\subseteq U$. By $(1)$, there is $X\in U$ such that $X\subseteq_{\mathcal{F}} X_i^c$ for each $i$. This means that $X\cap X_i\in \Fine(\kappa,\lambda)$ for every $i<\lambda$. The other direction is similar.   


\end{proof}
The variation of Galvin's theorem we obtain here is with respect to the following special kind of intersection:
\begin{definition}
    Let $\l A_x\mid x\in P_\kappa(\lambda)\r\subseteq P_\kappa(\lambda)$. The \textit{cone intersection} of the sequence $\l A_x\mid x\in P_\kappa(\lambda)\r\subseteq P_\kappa(\lambda)$, is the set $$\hourglass_{x\in P_\kappa(\lambda)}A_x=\{y\in P_\kappa(\lambda)\mid \forall x\in P_\kappa(\lambda),\ (x\prec y\vee y\subseteq x)\Rightarrow y\in A_x\}$$
\end{definition}
Clearly, the cone intersection satisfies:
$$\bigcap_{x\in P_\kappa(\lambda)}A_x\subseteq \hourglass_{x\in P_\kappa(\lambda)}A_x\subseteq\Delta_{x\in P_\kappa(\lambda)}A_x$$
For example, if $A_x=\{z\in P_\kappa(\lambda)\mid x\subseteq z\}$, then for every $y\in P_\kappa(\lambda)$, there is $y\subsetneq x$, which means that $y\notin A_x$, and in turn $y\notin \hourglass_{x\in P_\kappa(\lambda)}A_x$, i.e. $\hourglass_{x\in P_\kappa(\lambda)}A_x=\emptyset$. If follows that normal measures on $P_\kappa(\lambda)$ are not in general closed under cone intersections of their elements. Nonetheless, we have the following analogy of Galvin's theorem:
\begin{theorem}\label{thm: wierd Galvin}
    Suppose that $\lambda^{<\kappa}=\lambda$ and let $U$ be a normal fine ultrafilter over $P_\kappa(\lambda)$. Then for any $\l A_\alpha\mid \alpha<\lambda^+\r\subseteq U$ there are sets $H_x\in [\lambda^+]^{\lambda^+}$ for $x\in P_\kappa(\lambda)$ such that for every choice $\l\alpha_x\mid x\in P_\kappa(\lambda)\r\in \prod_{x\in P_\kappa(\lambda)}H_x$, $\hourglass_{x\in P_\kappa(\lambda)}A_{\alpha_x}\in U$.
\end{theorem}
\begin{proof}
    For each $x\in P_\kappa(\lambda)$ and $\alpha<\lambda^+$ let
    $$H_{\alpha,x}=\{\beta<\lambda^+\mid A_\beta\cap P(x)=A_\alpha\cap P(x)\}$$
    \begin{claim}
        There is $\alpha^*<\lambda^+$ such that for every $x\in P_\kappa(\lambda)$, $|H_{\alpha^*,x}|=\lambda^+$.
    \end{claim}
    \begin{proof}[Proof of Claim.]
        Otherwise, for each $\alpha<\lambda^+$ there is $x_\alpha$ such that $|H_{\alpha,x_\alpha}|\leq\lambda$. Since $\lambda^{<\kappa}=\lambda$, there is $S\in [\lambda^+]^{\lambda^+}$ and $x^*\in P_\kappa(\lambda)$ such that for every $\alpha\in S$, $x_\alpha=x^*$. Next, note that the value of $A_\alpha\cap P(x^*)$ has $2^{2^{|x^*|}}$-many possibilities. Since $\lambda^{<\kappa}=\lambda$, there is $A^*\subseteq P(x^*)$ and $S'\in [S]^{\lambda^+}$ such that for every $\alpha\in S'$, $A_\alpha\cap P(x^*)=A^*$. But this is impossible, since if $\alpha\in S'$, then on one hand $|H_{\alpha,x_\alpha}|\leq\lambda$, on the other hand, $S'\subseteq H_{\alpha,x_\alpha}$.
    \end{proof}
    Fix $\alpha^*$ as in the claim above and let $H_x=H_{\alpha^*,x}$. For every $x\in P_\kappa(\lambda)$, choose $\alpha_x\in H_{\alpha^*,x}$ and let $A_x=A_{\alpha_x}$. To see that $\hourglass_{x\in P_\kappa(\lambda)}A_x\in U$, we prove that $\id_U=j_U[\lambda]\in j_U(\hourglass_{x\in P_\kappa(\lambda)}A_x)=(\hourglass_{x\in P_{j(\kappa)}(j(\lambda))}A'_x))^{M_U}$, where $j_U(\l A_x\mid x\in P_\kappa(\lambda)\r)=\l A'_x\mid x\in P_{j(\kappa)}(j(\lambda))^{M_U}\r$. Towards this, following the definition of the cone intersection, let $x\prec j_U[\lambda]$, then there is $y\in P_\kappa(\lambda)$ such that $j_U(y)=j_U[y]=x$. In particular, $A'_x=j_U(A_y)$ and since $A_y\in U$, $\id_U\in j_U(A_y)$. For the second part of the definition of $(\hourglass_{x\in P_\kappa(\lambda)} A'_x)^{M_U}$, let $y\in (P_{j(\kappa)}(j(\lambda)))^{M_U}$ such that $\id_U\subseteq y$. By elementarity and the choice of the sequence $\l \alpha_x\mid x\in P_\kappa(\lambda)\r$, $M_U\models j_{U}(A_{\alpha^*})\cap P(y)=A_{y}\cap P(y)$. Since $\id_U\subseteq y$ we conclude that $\id_U\in j_U(A^*)\cap P(y)$ and therefore in $A_y$. It follows that $ \id_U\in j_U(\hourglass_{x\in P_{\kappa}(\lambda)}A_x)$, as wanted.
\end{proof}
\begin{remark}
    The above generalizes Galvin's theorem in the following sense: a normal ultrafilter $U$ on a measurable cardinal $\kappa$ can be identified as a normal fine ultrafilter on $P_\kappa(\kappa)$ which concentrates on the set of ordinals $\kappa\subseteq P_\kappa(\kappa)$.
    Now the above theorem says that if we take $\l A_\alpha\mid \alpha<\kappa^+\r\subseteq U$, then there are $H_x\in [\kappa^+]^{\kappa^+}$ such that for every $\l \alpha_x\mid x\in P_\kappa(\kappa)\r\in \prod_{x\in P_\kappa(\kappa)}$. 
    $\hourglass_{x\in P_\kappa(\kappa)}A_{\alpha_x}\in U$. Hence  $ \hourglass_{x\in P_\kappa(\kappa)}A_{\alpha_x}\in U$. 
    We claim that $\text{Card.}\cap\hourglass_{x\in P_\kappa(\kappa)}A_{\alpha_x}\subseteq \bigcap_{\nu<\kappa}A_{\alpha_\nu}$, where $\text{Card.}$ is the class of cardinals (here we view $\nu$ as an element of $P_\kappa(\kappa)$).
    Indeed, let $\rho\in \text{Card.}\cap \hourglass_{x\in P_\kappa(\kappa)}A_{\alpha_x}$ and $\nu<\kappa$. 
    Then either $\nu<\rho$ in which case $\nu\prec \rho$ and thus $\rho\in A_{\alpha_\nu}$. 
    Otherwise, $\rho\leq \nu$ in which case $\rho\subseteq \nu$ and again $\rho \in A_{\alpha_\nu}$ by definition of $\hourglass$.
    We conclude that $\rho\in \bigcap_{\nu<\kappa}A_{\alpha_\nu}$ which recovers Galvin's theorem.
\end{remark}
\subsection{An Ideal From Two-Cardinal Filter Games}
Throughout this section we assume that $\kappa \leq \lambda$ are cardinals and $\lambda^{<\kappa} = \lambda$. In this section, we would like to generalize Theorem~\ref{theorem: fmz} of \cite[Thm 1.2]{foremanmagidorzeman} connecting winning strategy in the filter games and large cardinal ideals, to the two-cardinal filter games. These were interested in \cite{tomvika}. To set up the game, we first need the following definition: 
\begin{definition}
    Fix some large regular cardinal $\theta$.
    A set $M \prec H_\theta$ of size $\lambda$ is called a \emph{$(\kappa,\lambda)$-model} if 
    \begin{enumerate}
        \item $M$ is transitive.
        \item $M \models \ZFC^{-}$. \ {}\footnote{$ZFC^{-}$ is the theory obtained from $\ZFC$ by removing the axiom of powerset, replacing the replacement schema with collection, and replacing the axiom of choice to the well-ordering principle}
        \item $M \prec_{\Sigma_0} V$.
        \item $\lambda+1 \subseteq M$ and $P_\kappa(\lambda)^M \subseteq M$.
    \end{enumerate}
    \end{definition}
    Given a $(\kappa,\lambda)$-model $N$ we say that $U$ is a normal $N$-ultrafilter on $P_\kappa(\lambda)$ if $$(N,\in,U)\models U\text{ is a fine normal ultrafilter on }P_\kappa(\lambda).$$ That is, $U$ measures all the sets in $P(P_\kappa(\lambda))\cap N$, and whenever $\l A_\alpha\mid \alpha<\lambda\r\in N$ is a collection of $\lambda$-many sets in $U$, $\Delta_{\alpha<\lambda}A_\alpha\in U$.

    Let us turn to the definition of the games $G_1$ and $G_2$. This is a game between two players, \textit{the Challenger} and \textit{the Judge}, taking turns. Informally, the Challenger presents the Judge with a challenge-- a collection of sets they have to measure. The Judge responds with a normal ultrafilter measuring this collection. In further steps of the game, the Challenger can add more sets to the ones which they previously presented, and the Judge has to extend the previous ultrafilter to the new collection.

    The game $G_1$, generalizing the game $G_1$ from \cite{foremanmagidorzeman}. For that we need an \emph{internally approachable} sequence, that is, a sequece $\l N_i : i< \lambda^+\r$ of $(\kappa,\lambda)$-models which is an increasing, continuous elementary chain such that $\l N_i : i < \alpha\r \in N_{\alpha'} $ for all $\alpha < \alpha' < \lambda^+$. 
\begin{remark}
    If $|P(P_\kappa(\lambda))| = \lambda^+$ then in $H_\theta$ there is a wellorder of $P_{\lambda^+}(P(P_\kappa(\lambda)))$ in order type $\lambda^+$.
    Hence $P(P_\kappa(\lambda)) = \bigcup_{i< \mu} N_i\cap P(P_\kappa(\lambda))$, and every $N_i$-ultrafilter appears in $N_j$ for some $j > i$.
\end{remark}
    \begin{definition}[The game $G_1$]
    Fix any ordinal $\gamma$. The rules of $G^\gamma_1$ are as follows:
    \begin{enumerate}
            \item The Challenger plays an increasing sequence of ordinals $\alpha_i < \lambda^+$.
            \item The Judge plays a sequence $U_i$ of $N_{\alpha_i+1}$-ultrafilters on $P_\kappa(\lambda)$  such that:
            \begin{enumerate}
                \item $\l U_j\mid j<i\r\in N_{\alpha_i+1}$.
                \item $U_j \subseteq U_i$ for all $i > j$.
            \end{enumerate}
        \end{enumerate}
    The Challenger goes first at limit stages.
    The game proceeds for some length $\ell < \gamma$ determined by the play, for some fixed $\gamma \leq \lambda^+$.
    The game continues until either the Judge has no valid move or the play has reached length $\gamma$.
    The Judge wins only if the play reaches length $\gamma$.
\end{definition}
%
\begin{definition}[The game $G_2$]
    Fix an internally approachable sequence of $(\kappa,\lambda)$-models $\l N_i : i<\lambda^+\r$.
    The rules of $G^\gamma_2$ are as follows:
        \begin{enumerate}
            \item The Challenger plays an increasing sequence of ordinals $\alpha_i < \lambda^+$.
            \item The Judge plays a sequence of sets $Y_i \subseteq P_\kappa(\lambda)$ such that:
            \begin{enumerate}
                \item $Y_i \subseteq_{\mathcal{F}} Y_j$ for all $i > j$ and $\l Y_j\mid j<i\r\in N_{\alpha_i+1}$.
                \item $U_i = \{ X \in P(P_\kappa(\lambda)) \cap N_{\alpha_i+1} : Y_i \subseteq_{\mathcal{F}} X\}$ is a normal $N_{\alpha_i+1}$-ultrafilter on $P_\kappa(\lambda))$.
            \end{enumerate}
        \end{enumerate}
        \end{definition}
    The rules and winning conditions of $G^\gamma_2$ are the same as those of $G^\gamma_1$.
\begin{proposition}
    The following are equivalent:
    \begin{enumerate}
        \item the Judge has a winning strategy in $G^\gamma_1$.
        \item the Judge has a winning strategy in $G^\gamma_2$.
    \end{enumerate}
\end{proposition}
\begin{proof}
    The direction $(2) \Rightarrow (1)$ is trivial, since given the set produced by a strategy for $G_2^\gamma$ the Judge can play the ultrafilter determined by that set in $G_1^\gamma$.

    Given a $G^\gamma_1$-winning strategy $\sigma$, we will define a $G^\gamma_2$-winning strategy $\sigma'$.
    For the first move, the Challenger plays an ordinal $\alpha_0$ and $\sigma$ will yield a normal $N_{\alpha_0}+1$ ultrafilter $U_0$.
    To determine the move $Y_0$ given by $\sigma'$, choose $\beta$ large enough that $U_0 \in N_\beta$ and choose a minimal enumeration of $U_0 = \{X_\xi : \xi < \lambda \}$ is the well-ordering of $H_\theta$. By elementarity, and since $U_0\in N_\beta$, this enumeration is in $N_\beta$.
    Then set $Y_0 = \triangle_{\xi < \lambda} X_\xi$. Note that any legal move $\alpha_1$ of the challenger will have to satisfy that $Y_0\in N_{\alpha_1}$ and also $U_0\in N_{\alpha_1}$. Again by elementarity, we will have that the minimal enumeration $U_0 = \{X_\xi : \xi < \lambda \}$ will also be in $N_{\alpha_1}$, so whenever $U_0\subseteq U_1$ is a normal $N_{\alpha_1}$-ultrafilter, $Y_0\in U_1$.
    
    Now suppose that $\l(\alpha_j,Y_j)\mid j<i\r$ is a run in the game according to $\sigma'$ which has already been defined. Also assume that we maintained an auxiliary $G^\gamma_1$-run $\l (\alpha_j,U_j)\mid j<i\r$, such that each $Y_j$ is the diagonal intersection of $U_j$ according to the minimal enumeration of $U_j$ in the $H_\theta$ well-ordering.

    Now suppose that the challenger in the game $G^\gamma_2$ plays a legal move $\alpha_i$. Note that $\l (U_j,Y_j)\mid j<i\r \in N_{\alpha_i+1}$, and that for each $j<i$, the minimal enumeration of $U_j$ is a member of $N_{\alpha_i+1}$. Next, we let the challenger in the game $G^\gamma_1$ play $\alpha_i$, and the strategy $\sigma'$ produces an $N_{\alpha_i+1}$-normal ultrafilter $U_i$ which extends all the $U_j$'s for $j<i$. It follows that $Y_j\in U_i$ for each $j<i$. Let $U_i=\{X^i_\xi\mid \xi<\lambda\}$ be the minimal enumeration of $U_i$, we set $Y_i=\triangle_{\xi < \lambda}X^i_\xi$. It is routine to check that $\sigma$ is a winning strategy in the game $G^\gamma_2$.
 \end{proof}
Next, we shall define some properties of ideals that appear in the main theorem of this section.
\begin{definition}
    Let $I$ be an ideal on $P_\kappa(\lambda)$.
    \begin{enumerate}
        \item $I$ is \emph{normal} if $I$ is closed under diagonal unions: whenever $\l X_\alpha : \alpha < \lambda\r \subseteq I$,
        \[ \underset{{\alpha<\lambda}}{\bigtriangledown}X_\alpha \coloneq \{ x \in P_\kappa(\lambda) : x \in \bigcup_{\alpha \in x} X_\alpha \}\in I\]
        \item $I$ is \emph{precipitous} if whenever $G \subseteq P(P_\kappa(\lambda))/I$ is generic over $V$, the generic $V$-ultrapower $\Ult(V,G)$ is well-founded.
        \item $D \subseteq I^+$ is \emph{dense} if $D$ is a dense subset of the partial order $(P(P_\kappa(\lambda))/I,\subseteq_I)$.
        \item $I$ is \emph{$\lambda$-measuring} if for any $\l A_\alpha\mid \alpha<\lambda\r\subseteq P_\kappa(\lambda)$ and any $S\in I^+$ there is $S'\subseteq S$, $S'\in I^+$ such that for any $\alpha<\lambda$, $S'\subseteq_{\mathcal{F}} A_\alpha$ or $S'\subseteq_{\mathcal{F}} P_\kappa(\lambda)\setminus A_\alpha$. 
    \end{enumerate}
\end{definition}

\begin{theorem}\label{theorem: game ideal}
    Suppose that $2^\lambda=\lambda^+$, there is no saturated ideal on $P_\kappa(\lambda)$, and that the Judge has a winning strategy for the game $G^\gamma_2$ for some regular $\omega<\gamma\leq\lambda$.
    Then there is an ideal $I$ on $P_\kappa(\lambda)$ such that:
    \begin{enumerate}
        \item $I$ is normal.
        \item $I$ is precipitous.
        \item $I^+$ has a dense subtree $(T,\supseteq_\mathcal{F})$ which is $\gamma$-closed.
        \item $I$ is $\lambda$-measuring. 
    \end{enumerate}
\end{theorem}
\begin{proof}
    Given a winning strategy $\sigma$ for the Judge, we define the \emph{hopeless ideal} $I(\sigma)$ as the collection of all $X \subseteq P_\kappa(\lambda)$ such that no play of the game played according to $\sigma$ ends with a $Y \subseteq_{\mathcal{F}} X$. In other words, no ultrafilter generated in any play according to $\sigma$ assigns measure one to $X$.
    
    More precisely, for any $\ell \leq \gamma$ we recursively define a play of $G^\ell_2$ according to $\sigma$ as a sequence $P = \l (\alpha_i, Y_i) : \sigma((P \upharpoonright i)^{\smallfrown}\alpha_i) = Y_i \text{ for } i < \ell \r$. Now define the hopeless ideal:
    \[I(\sigma) = \{ X \subseteq P_\kappa(\lambda) : \text{for any play $P$ and any $\alpha$, } \sigma(P{}^\smallfrown \alpha) \not\subseteq_{\mathcal{F}} X \}\]
    
    We also define the \emph{conditional hopeless ideal} given some play $P$ of length $\ell < \gamma$:
    \[I(\sigma,P) = \{ X \subseteq P_\kappa(\lambda) : \text{for any play $Q \sqsupseteq P$ and any $\alpha$, } \sigma(Q{}^\smallfrown\alpha) \not\subseteq_{\mathcal{F}} X \}\]
    \begin{claim}
        $I(\sigma)$ is $\kappa$-complete and normal.
    \end{claim}
    \begin{claimproof}
        First we check $\kappa$-completeness.
        Suppose towards a contradiction that $I(\sigma)$ is not $\kappa$-complete, so there is some $\eta < \kappa$ and $\l A_\xi : \xi < \eta \r \subseteq I(\sigma)$ such that $A = \bigcup_{\xi < \eta} A_\alpha \notin I(\sigma)$.
        Then there must be some play $P$ of length $\ell$ such that $Y_i \subseteq_{\mathcal{F}} A$ 
        for some $Y_i$ appearing in $P$.
        Now we can choose some play $Q$ such that $P \res i = Q \res i$ and $\l A_\xi : \xi < \eta \r \subseteq N_{\alpha_{i+1}}$. 
        In particular, $Y_{i+1} \subseteq_{\mathcal{F}} Y_{i} \subseteq_{\mathcal{F}} A$ and for all $\xi < \eta$, $\bigcup_{j \leq i+1} A_j \subseteq_{\mathcal{F}} A_\xi$.
        Hence $Y_{i+1} \subseteq_{\mathcal{F}} A_\xi$ by normality of $U_i$, so there is some $A_\xi \notin I(\sigma)$, a contradiction.

        The same proof with $\bigtriangledown_{\xi < \lambda} A_\xi$ in place of $\bigcup_{\xi < \eta} A_\xi$ shows that $I(\sigma)$ is normal.
        The same proof also shows that $I(\sigma,P)$ is $\kappa$-complete and normal for any $P$.
    \end{claimproof}
    
    To construct the desired ideal, we will first take an arbitrary winning strategy $\sigma$ for II in $G^\gamma_2$ and build the tree $T(\sigma)$ together with a correspondence taking sets $X$ in $T(\sigma)$ to plays of the game $R_X$. 
    We will then use $T(\sigma)$ to construct a new strategy $\sigma'$ for which $T(\sigma)$ will witness the required property $(3)$ of $I(\sigma')$.

    The construction of $T(\sigma)$ and the assignment $X\mapsto R_X$ will go by induction on the levels of $T(\sigma)$. 
    We wish to maintain the following in our induction:
    \begin{enumerate}
        \item $R_X$ has successor length and ends with $X$ as the last move played by the Judge.
        \item The tree order is  $\subseteq_{\mathcal{F}}$ restricted to the nodes of the tree.
        \item If $X \supseteq_{\mathcal{F}}Y$ are nodes in $T(\sigma)$ then $R_X \sqsubseteq R_Y$.
        \item If $X,Y$ are on the same level of the tree then $X \cap Y \in \Fine(\kappa,\lambda)^*$.
    \end{enumerate}
    
    For $\delta$ limit we will take $T(\sigma) \res \delta = \bigcup_{\alpha < \delta} (T(\sigma) \res \alpha)$ and $R_X$ will be already be defined for every $X \in T(\sigma) \res \delta$. The induction hypothesis is trivially maintained. So we shall focus on the successor stage of the construction.

    Let $\delta < \gamma$ and suppose we have constructed every level up to $\delta$.
    Let $b$ be a cofinal branch of the tree constructed thus far.
    Let $P_b=\bigcup_{X\in b}R_X$. 
    Then by $(2)$, $P_b$ is a play according to $\sigma$. 
    Since we assume there is no saturated ideal, we can choose some antichain $\mathcal{A}_b \subseteq I(\sigma,P_b)^+$ with $|\mathcal{A}| = \lambda^+$.
    For each $A \in \mathcal{A}_b$ let $Q_A$ be a play extending $P_b$ such that if $Y_A$ is the last move of the Judge then $Y_A \subseteq_{\mathcal{F}} A$. Such a play exists by the definition of $I(\sigma,P_b)$.

    Now we construct notes at the $\delta^{\text{th}}$-level, which will all have $b$ as their set of predecessors in the tree. This is again by recursion.
    Suppose we have constructed the successors $\l Y_\xi : \xi < \bar{\xi}\r$, and let $\alpha(Y_\xi)$ denote the index in $R_{Y_\xi}$ played by the Challenger to which $Y_\xi$ was the response to.
    Since each $N_\beta$ has size $\lambda$ we may choose the some $A \in \mathcal{A}_b$ which is not an element of any $N_{\alpha(Y_\xi)+1}$ for any $\xi < \bar{\xi}$.
    Now let $\alpha$ be least such that $A$ and every $Y_\xi$ is in $N_{\alpha+1}$.
    Then $\alpha$ is a legal move following $Q_A$, so we may let $Y_{\bar{\xi}}$ be the Judge's response to $Q_A {}^\frown \l \alpha\r$ given by $\sigma$.
    Notice that $Y_\xi \subseteq_{\mathcal{F}} A$ as $A$ was given by a previous move of the Judge.
    Also, set $R_{Y_{\bar{\xi}}} = Q_A {}^\frown (\alpha, Y_{\bar{\xi}})$.

    Now we shall show that the three conditions of the inductive hypothesis are maintained.
    The first is easily seen to be satisfied, so towards verifying the second condition, et $X$ be any node in $T(\sigma)$ and $Y_\xi$ any node on the newly constructed $\delta^{\text{th}}$ level of the tree. If $X$ is a predecessor of $Y_\xi$ is the tree order, then $X$ lie on the branch $b$, and therefore
    $Y_\xi\subseteq_{\mathcal{F}} X$ as $Y_\xi$ was given by a move of the Judge to a run of the game where $X$ was played. In the other direction, if $X$ is not in $b$, then either $X=Y_{\bar{\xi}}$ in which case
    $Y_\xi\not\subseteq_{\mathcal{F}} X$, or the predecessors of $X$ and $Y_\xi$ in the tree are different. Applying $(4)$ of the induction hypothesis to the level where $X$ and $Y_\xi$ split, together with $(2)$, we see that $Y_\xi\not\subseteq_\mathcal{F} X$. 
    
     The third condition now follows from the second, since if let $X\subseteq Y_\xi$, then $X$ must be on the branch $b$ leading to $Y_\xi$. But then, by definition, $$R_X\sqsubseteq b\sqsubseteq Q_A\sqsubseteq R_{Y_\xi},$$ where $A\in\mathcal{A}_b$ was the set used in the construction of $Y_\xi$.
     
    Finally, we verify the fourth condition.
    $Y,Z$ lie above distinct branches then this follows from the inductive hypothesis. So suppose now that $\xi < \zeta$ and $Y_\xi,Y_\zeta$ are above the same branch $b$.
    Let $A_\xi,A_\zeta$ be the corresponding members of $\mathcal{A}_b$, respectively.
    Then $Y_\xi  \subseteq_\mathcal{F} A_\xi$ and similar for $\zeta$.
    But $A_\xi \cap A_\zeta \in \Fine(\kappa,\lambda)^*$ since $\mathcal{A}_b$ is an antichain, so $Y_\xi \cap Y_\zeta \in \Fine(\kappa,\lambda)^*$.
    Hence the induction hypothesis is maintained and so the construction of $T(\sigma)$ is completed.

    Now given $T(\sigma)$ we shall recursively define a new strategy $\sigma'$.
    Suppose $R = \l (\alpha_i, Y_j): i < j \r$ is a play of the game according to $\sigma' \res j$, and assume further that every move given by $\sigma'$ is in $T(\sigma)$.
    Let $b_R$ be the corresponding branch through $T(\sigma)$.
    For each legal move $\beta$ for the Challenger, set
    \begin{align*}
        \text{$\sigma'(R^\frown \l\beta\r)$} &= \text{ the unique immediate successor $Y$ of $b_R$ in} \\
        & \text{ \hspace{5mm}$T(\sigma)$ such that $\alpha(Y) >\beta$ is minimal.}
    \end{align*}
    Since $\sigma$ is a winning strategy and $T(\sigma)$ has height $\gamma$ we see that $\sigma'$ is also a winning strategy.
    Also since $T(\sigma)$ has height $\gamma$ and no terminal nodes we see that $T(\sigma)$ is a $\gamma$-closed subset of $I(\sigma')^+$.


    We have seen already that $I(\sigma')$ is normal, so we show it is precipitous.
    Consider the precipitousness game on an ideal $J$, which is played as follows:
    players I and II alternate picking sets $S_n \in J^+$ such that $S_{n+1} \subseteq S_n$ for all $n \in \omega$.
    The game is played for $\omega$ steps.
    I wins iff $\bigcap_{n \in \omega} S_n = \emptyset$.
    A proof that $I$ is precipitous iff I does not have a winning strategy can be found in \cite[Lemma 22.21]{jech}.
    \begin{claim}
        Player I does not have a winning strategy in the precipitousness game $G(I(\sigma'))$, so $I(\sigma')$ is precipitous.
    \end{claim}
    \begin{claimproof}
        Assume that $\tau$ is a winning strategy for I in the precipitousness game $G(I(\sigma'))$.
        We will show that there is a play of the game according to $\tau$ where II wins.
        To do this, we shall use an auxiliary play of $G^\gamma_2$.
        
        First let $S_0$ be the first move for I given by $\tau$.
        Let $\alpha$ be least such that $S_0 \in N_{\alpha}$. Since $S_0\in I(\sigma')^+$, there is a run of the game $R_0$ ending with $Y_0$ such that $Y_0\subseteq_{\mathcal{F}}S_0$. Note that by the definition of $\sigma'$, $R_0=R_{Y_0}$. Back in
         the precipitousness game, we let II play $S_1=Y_0\cap S_0$. In general, suppose that $S_0\supseteq S_1\supseteq...\supseteq S_{2n-1}$ was played according to $\tau$ and let $S_{2n}\in I(\sigma')^+$ be the move of Player $I$ according to $\tau$. Again, note that there is a Run $R_n$ with last move $Y_n$ such that $Y_n\subseteq_{\mathcal{F}} S_{2n}$. Again, $R_n=R_{Y_n}$ and note that $Y_n\subseteq_{\mathcal{F}} S_{2n}\subseteq_{\mathcal{F}}Y_{n-1}$. By the properties of the tree, it follows that $R_{n-1}=R_{Y_{n-1}}\sqsubseteq R_{Y_n}=R_n$. Let $S_{2n+1}=Y_n\cap S_{2n}$ by the move of $II$. 


        Then at the end of the game, $\bigcap_{i<\omega}S_i=\emptyset$, since $\tau$ is a winning strategy. However, the run $R=\bigcup_{n<\omega} R_n$ is a run in the game $G^\xi_2$ according to the winning strategy $\sigma$. We can pick $\alpha$ large enough so it is a legal move and $\l S_n\mid n<\omega\r\i N_{\alpha+1}$. Then $\sigma(R^\smallfrown\alpha)=Y$ and $Y$ determines a normal ultrafilter $U$ on $N_{\alpha+1}$ which includes all the $Y_n$'s and therefore all the $S_n$'s. Also $U$ is $\kappa$-complete which implies that $\bigcap_{n<\omega}S_n\in U$, producing a contradiction.
    \end{claimproof}
    
    We remark that all we have really used in the above proof is that we have a winning strategy in the game of length $\omega+1$-closed.
    Now we just have one claim left to show.
    \begin{claim}
        $I(\sigma')$ is $\lambda$-measuring.
    \end{claim}
    \begin{claimproof}
        Fix $A \in T(\sigma)$ and a sequence $\l A_\alpha : \alpha < \lambda \r$.
        Let $\xi < \lambda^+$ be large enough that $\l A_\alpha : \alpha < \lambda \r \subseteq N_\xi$ 
        and let $A^* \subseteq_{\mathcal{F}} A_\alpha$ for all $\alpha$ with $A^* \in T(\sigma)$. 
        Then $A^*$ is a valid move for the Judge in response to the Challenger playing $\xi$.
        Since $I(\sigma')$ extends the fine ideal, $A^* \subseteq_{\mathcal{F}} A$ and $A^*$ measures $\l A_\alpha : \alpha < \lambda \r$.
    \end{claimproof}
    
    Hence $I(\sigma')$ is as desired.
\end{proof}
\begin{remark}
    Similar to Theorem \cite[Thm. 8.14]{tomvika}, the existence of a normal $\lambda$-measuring ideal
$I$ on $P_\kappa(\lambda)$ with a $\delta$-closed dense tree $D$. Implies that the Judge has a winning strategy in the game $G^\gamma_2$.
\end{remark}
\subsection*{Acknowledgment}
The authors would like to thank Gabriel Goldberg for proposing many ideas and directions that influenced this paper enormously. They would also like to thank James Cummings and Moti Gitik for insightful discussions on the matter. 
\printbibliography
\end{document}